\documentclass[12pt]{article}
\title{Forking and Dividing in Fields with Several Orderings and Valuations}
\author{Will Johnson}

\usepackage{amsmath, amssymb, amsthm}    	
\usepackage{fullpage} 	
\usepackage{graphicx}	
\usepackage{amscd}
\usepackage{hyperref}
\usepackage{centernot}
\usepackage{titlefoot}

\DeclareMathOperator*{\forkindep}{\raise0.2ex\hbox{\ooalign{\hidewidth$\vert$\hidewidth\cr\raise-0.9ex\hbox{$\smile$}}}}

\newcommand{\concatA}{%
  \mathord{
    \mathchoice
    {\raisebox{1ex}{\scalebox{.7}{$\frown$}}}
    {\raisebox{1ex}{\scalebox{.7}{$\frown$}}}
    {\raisebox{.7ex}{\scalebox{.5}{$\frown$}}}
    {\raisebox{.7ex}{\scalebox{.5}{$\frown$}}}
  }
}

\newcommand{\Abs}{\operatorname{Abs}}

\newcommand{\bdn}{\operatorname{bdn}}
\newcommand{\diag}{\operatorname{diag}}
\newcommand{\qftp}{\operatorname{qftp}}
\newcommand{\rk}{\operatorname{rk}}

\newcommand{\Gal}{\operatorname{Gal}}

\newcommand{\Frac}{\operatorname{Frac}}

\newcommand{\acl}{\operatorname{acl}}
\newcommand{\dcl}{\operatorname{dcl}}
\newcommand{\tp}{\operatorname{tp}}
\newcommand{\stp}{\operatorname{stp}}

\newcommand{\val}{\operatorname{val}}

\newcommand{\trdeg}{\operatorname{tr.deg}}

\newtheorem{theorem}{Theorem}[section] 
\newtheorem{lemma}[theorem]{Lemma}

\newtheorem{corollary}[theorem]{Corollary}
\newtheorem{fact}[theorem]{Fact}

\newtheorem{proposition-eh}[theorem]{Proposition(?)}
\newtheorem*{theorem-star}{Theorem}
\newtheorem*{conjecture-star}{Conjecture}
\newtheorem*{lemma-star}{Lemma}

\theoremstyle{definition}
\newtheorem{definition}[theorem]{Definition}

\theoremstyle{remark}
\newtheorem{remark}[theorem]{Remark}
\newtheorem{claim}[theorem]{Claim}

\newtheorem*{acknowledgment}{Acknowledgments}

\newcommand{\Rr}{\mathbb{R}}

\newcommand{\Zz}{\mathbb{Z}}

\newcommand{\ACVF}{\mathrm{ACVF}}
\newcommand{\RCF}{\mathrm{RCF}}
\newcommand{\pCF}{p\mathrm{CF}}
\newcommand{\ACF}{\mathrm{ACF}}

\newenvironment{claimproof}[1][\proofname]
               {
                 \proof[#1]
                 
               }
               {
                 \endproof
               }

\begin{document}
\maketitle\unmarkedfntext{

}

\begin{abstract}
  We consider existentially closed fields with several orderings,
  valuations, and $p$-valuations.  We show that these structures are
  NTP$_2$ of finite burden, but usually have the independence
  property.  Moreover, forking agrees with dividing, and forking can
  be characterized in terms of forking in ACVF, RCF, and $p$CF.
\end{abstract}

\section{Introduction}
Consider the theory of fields with $n$ distinct valuations.  By the
thesis of van den Dries \cite{LvdD}, this theory has a model
companion.  More generally, one can add orderings and $p$-valuations
into the mix, and a model companion exists.  We will explore the
classification-theoretic properties of this model companion.

To be more precise, suppose that for each $1 \le i \le n$\ldots
\begin{itemize}
\item The theory $T_i$ is one of ACVF (algebraically closed valued
  fields), RCF (real closed fields), or $p$CF ($p$-adically closed
  fields).
\item $\mathcal{L}_i$ is some language in which $T_i$ has quantifier
  elimination, such as the language of ordered rings for RCF, and the
  Macintyre language for $p$CF.
\item $(T_i)_\forall$ is the universal fragment of $T_i$, plus the
  field axioms.  For example, $\RCF_\forall$ is the theory of ordered
  fields, and $\ACVF_\forall$ is the theory of valued fields.
\end{itemize}
Arrange that $\mathcal{L}_i \cap \mathcal{L}_j = \mathcal{L}_{rings}$
for $i \ne j$, and form the theory $\bigcup_{i = 1}^n (T_i)_\forall$.
In van den Dries's notation, this theory is denoted $((T_1)_\forall,
(T_2)_\forall, \ldots, (T_n)_\forall)$.
For example
\[ (\ACVF_\forall, \ACVF_\forall, \ACVF_\forall)\]
is the theory of fields with three distinct valuations.  The $T_i$ can
be mixed; for example
\[ (\ACVF_\forall, \RCF_\forall, 3\mathrm{CF}_\forall)\]
is the theory of fields with a valuation, an ordering, and a
3-valuation (plus Macintyre predicates).

In all these cases, van den Dries proves the existence of a model companion
\[ \overline{((T_1)_\forall, (T_2)_\forall, \ldots, (T_n)_\forall)}.\]
In fact, van den Dries's result is more general than what we have
stated, allowing the $T_i$'s to be arbitrary theories with quantifier
elimination such that the $(T_i)_\forall$ are ``t-theories''
(Definition III.1.2 in \cite{LvdD}).

However, we will only consider the case where the $T_i$ are ACVF, RCF,
or $p$CF.  In these cases, we will prove the following about the model
companion $\overline{((T_1)_\forall, \ldots, (T_n)_\forall)}$, which
we denote $T$ for simplicity:
\begin{enumerate}
\item \label{z1} $T$ is NTP$_2$, but fails to be NIP (or NSOP) when $n > 1$.  See
  Theorems~\ref{thm-ntp2} and \ref{thm-nip}.  If $n = 1$, then $T$ is
  one of ACVF, RCF, or $p$CF, which are all known to be NIP.
\item Moreover, $T$ is ``strong'' in the sense of Adler~\cite{Adler},
  and, every type has finite burden.  The burden of affine $m$-space
  is exactly $mn$, where $n$ is the number of valuations and
  orderings.  See Theorem~\ref{thm-ntp2}.
\item \label{z3} Forking and dividing agree over sets in the home sort, so every
  set in the home sort is an ``extension base for forking'' in the
  sense of Chernikov and Kaplan \cite{CK}.  See
  Theorem~\ref{forking-is-dividing}.
\item \label{z4} Forking in the home sort has the following
  characterization (Theorem~\ref{thm-char}).  Suppose $K \models T$,
  and $A, B, C \subseteq K$ are subsets of the home sort.  For $1 \le
  i \le n$, let $K_i$ be a model of $T_i$ extending the
  $\mathcal{L}_i$-reduct of $K$.  For example, in the case of $n$
  orderings, $K_i$ could be a real closure of $K$ with respect to the
  $i^{\text{th}}$ ordering.  Then $A \forkindep_C B$ holds in $K$ if
  and only if $A \forkindep_C B$ holds in $K_i$ for every $i$.  The
  choice of the $K_i$ does not matter.
\end{enumerate}
It is likely that (\ref{z3}) also holds of sets of imaginaries, which
would imply that Lascar strong type and compact strong type agree, by
\cite{BYC} Corollary 3.6.


In the case where every $T_i$ is ACVF, Theorem~\ref{miracle} gives a
simple axiomatization of the model companion $T$: a model of $T$ is
simply an algebraically closed field with $n$ independent non-trivial
valuations.  In this case, forking is characterized as follows: $A
\forkindep_C B$ holds if and only if it holds in the reduct $(K,v_i)$,
for every $i$.

Something similar happens when all but one of the $T_i$ is ACVF.  For
example, if $T_1$ is RCF and $T_2, \ldots, T_n$ are all ACVF, then a
model of $T$ is a real closed field with $n-1$ valuations such that
the ordering and the valuations are pairwise independent.  See
Theorem~\ref{miracle} for details.

As a concrete example, let $K$ be one of the following fields:
$\mathbb{F}_p(t)^{alg}$, $\mathbb{Q}^{alg}$, $\mathbb{Q}^{alg} \cap
\mathbb{R}$, or $\mathbb{Q}^{alg} \cap \mathbb{Q}_p$ for some $p$.
Let $R_1, \ldots, R_n$ be valuation rings on $K$.  Then $K$ with the
ring structure and with a unary predicate for each $R_i$ is a strong
NTP$_2$ theory of finite burden, and every set of real elements is an
extension base.  The same holds for $\bigcap_{i = 1}^n R_i$ as a pure
ring.

\subsection{Related and future work}
Existentially closed fields with several orderings were independently
shown to be NTP$_2$ in Montenegro's thesis \cite{samaria-thesis}.
More generally, she shows that bounded pseudo-real-closed fields are
NTP$_2$, proving Conjecture~5.1 of \cite{CKS}.  Similarly, Montenegro
shows that bounded pseudo-$p$-adically-closed fields are NTP$_2$,
which includes the case where every $T_i$ is $p$CF.

The techniques of the present paper have been generalized in
\cite{silly-sequel} to prove that NTP$_2$ holds in algebraically
closed fields with several valuations.  The case of independent
valuations is Corollary~\ref{mir-ntp2} below.

Halevi, Hasson, and Jahnke use an argument related to
\S\ref{sec-miracle} and \S\ref{sec:nip} in order to prove that a field
with two independent valuations cannot be NIP if one of the two
valuations is henselian, which helps connect two
conjectures on the classification of NIP fields \cite{hhj-v-top}.

The classification of dp-minimal fields \cite{arxiv-myself} is not
directly related to the present paper, but suggests a direction for
future research.  Most of the properties shared by ACVF, RCF, and
$p$CF are shared by all dp-minimal theories of valued fields and
ordered fields.  Consequently, the list
\[ \ACVF, \RCF, \pCF\]
appearing throughout the present paper can probably be extended to
include all dp-minimal theories of valued fields and ordered fields.
But there are a large number of details to check.

\subsection{Conventions}
The algebraic closure of a field $K$ will be denoted $K^{alg}$.  A
variety over a field $K$ is a reduced finite-type scheme over $K$.  If
$V$ is a $K$-variety, $\dim V$ denotes the dimension of the definable
set $V(K^{alg})$ in the structure $K^{alg}$, rather than the definable
set $V(K)$ in the structure $K$.  For example, if $V$ is the
$\Rr$-variety cut out by the equations $x^2 + y^2 = 0$, then $\dim V$
is 1, not 0.

The monster model will be denoted $\mathfrak{M}$.  Forking
independence will be denoted $A \forkindep_C B$.  When working in a
field $K$, algebraic independence will be denoted $A \forkindep^\ACF_C
B$.  In other words, $A \forkindep^\ACF_C B$ means that $A
\forkindep_C B$ holds in $K^{alg}$.

When working with fields with several valuations and orderings, we
will generally use the following conventions:
\begin{itemize}
\item $T_i$ will denote the theory ACVF, RCF, or $p$CF.
\item $\mathcal{L}_i$ will denote the language of $T_i$.
\item $(T_i)_\forall$ will denote the universal fragment of $T_i$,
  plus the field axioms.
\item $T^0$ will denote the theory $\bigcup_{i = 1}^n (T_i)_\forall$.
\item $T$ will denote the model companion of $T^0$.
\end{itemize}

\subsection{Outline}
In Section~\ref{sec-misc}, we recall some elementary facts about ACVF,
$p$CF, and RCF which will be needed later.  In Section~\ref{review},
we quickly reprove the main facts needed from Chapters II and III of
van den Dries's thesis, arriving at a slightly different way of
expressing the axioms of the model companion, and handling the case of
positive characteristic, which was not explicitly considered by van
den Dries.  Section~\ref{sec-miracle} is a digression aimed at proving
Theorem~\ref{miracle}, which drastically simplifies the axioms of the
model companions in some cases.  Theorem~\ref{miracle} is probably
known to experts, but we include a proof here for lack of a reference.
In Section~\ref{probabilities}, we construct some Keisler measures
that will be used in the later sections.  In
Section~\ref{classification}, we determine where the model companion
lies in terms of various classification theoretic boundaries, proving
that it is NTP$_2$ and strong, but not NSOP and usually not NIP.  In
Section~\ref{forking}, we show that forking and dividing agree over
sets in the home sort, and we characterize forking in terms of forking
in the $T_i$'s.

\section{Various facts about ACVF, $p$CF, and RCF}\label{sec-misc}
Let $T$ be one of ACVF, RCF, or $p$CF.  Work in the usual one-sorted
languages with quantifier elimination:
\begin{itemize}
\item For ACVF, work in the language of fields with a binary predicate for
  $\val(x) \ge \val(y)$.
\item For RCF, work in the language of ordered rings.
\item For $p$CF, work in the Macintyre language with unary predicates
  for $n$th powers \cite{macintyre}.
\end{itemize}
Quantifier-elimination implies the following:
\begin{fact}\label{the-fact}
Let $\mathfrak{M}$ be a model of $T$, and $K$ be a subfield.  Every $K$-definable set is a positive boolean combination of topologically open sets and affine varieties defined over $K$.  In particular, any $K$-definable subset of $\mathfrak{M}^n$ has non-empty interior or is contained in a $K$-definable proper closed subvariety of $\mathbb{A}^n$.
\end{fact}

Let $\mathfrak{M}$ be a monster model of $T$.
\begin{definition}
Let $K$ be a subfield of $\mathfrak{M}$.  Let $D \subseteq \mathfrak{M}^n$ be a definable set, defined over $K$.  Define the \emph{rank} $\rk_K D$ to be
the supremum of $\trdeg(\alpha/K)$ as $\alpha$ ranges over $D$.
\end{definition}

\begin{lemma}\label{various}~
\begin{description}
\item[(a)]
If $D \subseteq \mathfrak{M}^n$, then $\rk_K D = n$ if and only if $D$ has non-empty interior.
\item[(b)]
If $D \subseteq \mathfrak{M}^n$ and $1 \le k \le n$, then $\rk_K D \ge k$ if and only if $\rk_K \pi(D) = k$ for one of the (finitely many) coordinate projections $\pi : \mathfrak{M}^n \twoheadrightarrow \mathfrak{M}^k$.
\item[(c)]
The rank of $D$ does not depend on the choice of $K$, and rank is definable in families.
\item[(d)]
If $D \subseteq V$ where $V$ is geometrically irreducible, then $\rk D = \dim V$ if and only if $D(\mathfrak{M})$ is Zariski dense in $V(\mathfrak{M}^{alg})$.
\end{description}
\end{lemma}
\begin{proof}
\begin{description}
\item[(a)] If $\rk_K D < n$, then every tuple $\alpha$ from $D$ lives inside an affine $K$-variety of dimension less than $n$.  By compactness, $D$ is contained in the union of finitely many affine $K$-varieties of dimension less than $n$.  This union contains the Zariski closure of $D$, so $D$ is not Zariski dense.  This forces $D$ to have no topological interior, because non-empty polydisks in affine space are Zariski dense.  Conversely, if $D$ has no interior, then by Fact~\ref{the-fact}, $D \subseteq V$ for some proper subvariety $V \subsetneq \mathbb{A}^n$ with $V$ defined over $K$.  Then $\rk_K D \le \dim V < n$.
\item[(b)] Clear by properties of rank in pregeometries.
\item[(c)] Combine (a) and (b).
\item[(d)] If $\rk D < \dim V$, then every point in $D$ is contained in an affine $K$-variety of dimension less than $\dim V$.  By compactness, $D$ is contained in the union of finitely many such varieties.  This finite union contains the Zariski closure of $D$, and is strictly smaller than $V$ itself.  Conversely, suppose that $D$ is not Zariski dense in $V$.  Let $V' \subsetneq V$ be the Zariski closure of $D$.  As $V$ is geometrically irreducible, $\dim V' < \dim V$.  Also, $V'$ is defined over $\mathfrak{M}$ rather than $\mathfrak{M}^{alg}$, because it is the Zariski closure of a set of $\mathfrak{M}$-points.  Let $L$ be a small subfield of $\mathfrak{M}$ over which $V'$ and $D$ are defined.  Then
\[ \rk_K D = \rk_L D \le \rk_L V' \le \dim V' < \dim V. \qedhere\]
\end{description}
\end{proof}

\begin{corollary}\label{extend-types-independently}
If $K \le L$ is an inclusion of small subfields of $\mathfrak{M}$ and $\alpha$ is a finite tuple, we can find $\alpha' \equiv_K \alpha$ with $\trdeg(\alpha'/L) = \trdeg(\alpha'/K)$.
\end{corollary}
\begin{proof}
Let $n = \trdeg(\alpha/K)$.  Let $\Sigma(x)$ be the partial type
asserting that $x \equiv_K \alpha$ and that $x$ belongs to no
$L$-variety of dimension less than $n$.  We claim that $\Sigma(x)$ is
consistent.  Otherwise, there is some formula $\phi(x)$ from
$\tp(\alpha/K)$ and some $L$-varieties $V_1, \ldots, V_m$ of dimension
less than $n$, such that $\phi(\mathfrak{M}) \subseteq \bigcup_{i =
  1}^m V_i$.  But then
\[ \rk_K \phi(\mathfrak{M}) = \rk_L \phi(\mathfrak{M}) \le \max_{1 \le i \le m} \dim V_i < n,\]
contradicting the fact that $\alpha \in \phi(\mathfrak{M})$ and $\trdeg(\alpha/K) \ge n$.

Thus $\Sigma(x)$ is consistent.  If $\alpha'$ is a realization, then $\alpha' \equiv_K \alpha$ and
\[ \trdeg(\alpha'/L) \ge n = \trdeg(\alpha/K) = \trdeg(\alpha'/K) \ge \trdeg(\alpha'/L). \qedhere\]
\end{proof}

\begin{corollary}\label{free-amalgamation}
Let $L$ and $L'$ be two fields satisfying $T_\forall$, and suppose they share a common subfield $K$.  Then $L$ and $L'$ can be amalgamated over $K$ in a way which makes $L$ and $L'$ be algebraically independent over $K$.
\end{corollary}
\begin{proof}
By quantifier elimination, we may as well assume that $L$ and $L'$ and $K$ live inside a monster model $\mathfrak{M} \models T$.  By Corollary~\ref{extend-types-independently} and compactness, we can extend $\tp(L/K)$ to $L'$ in such a way that any realization is algebraically independent from $L'$ over $K$.
\end{proof}

\begin{definition}\label{nearly}
Let $K \le L$ be an inclusion of fields.  Say that $K$ is \emph{relatively separably closed} in $L$ if every $x \in L \cap K^{alg}$ is in the perfect closure of $K$.
\end{definition}
This is a generalization of $K$ being relatively algebraically closed in $L$; in characteristic zero these two concepts are the same.
Note that if we embed $L$ into a monster model $\mathfrak{M}$ of ACF, then $K$ is relatively separably closed in $L$ if and only if $\dcl(K) = \acl(K) \cap \dcl(L)$ if and only if $\tp(L/K)$ is stationary.  From this, one gets
\begin{fact}\label{cb-stationarity-fact}
Let $L \ge K \le L'$ be (pure) fields.  Suppose that $K$ is relatively separably closed in $L$ or $L'$.  Then there is only one way to amalgamate $L$ and $L'$ over $K$ in such a way that $L$ and $L'$ are algebraically independent over $K$.
\end{fact}
\begin{fact}\label{cb-stationarity-variety}
If $K$ is relatively separably closed in $L$ and $\alpha$ is a tuple from $L$, and $V$ is the variety over $K$ of which $\alpha$ is the generic point, then $V$ is geometrically irreducible.
\end{fact}

\subsection{Dense formulas}
In this section, $T$ continues to be one of ACVF, RCF, or $p$CF.
\begin{definition}
Let $K$ be a model of $T_\forall$.  Let $V$ be a geometrically irreducible affine variety defined over $K$.  Let $\phi(x)$ be a quantifier-free formula with parameters from $K$, defining a subset of $V$ in any/every model of $T$ extending $K$.  Say that $\phi(x)$ is \emph{$V$-dense} if $\rk \phi(\mathfrak{M}) = \dim V$.
Here $\mathfrak{M}$ is a monster model of $T$ extending $K$.
\end{definition}
The choice of $\mathfrak{M}$ is irrelevant by quantifier-elimination in $T$ and by Lemma~\ref{various}(c).

\begin{lemma}\label{tfae-variant}
Let $K$ be a model of $T_\forall$, $L$ be a model of $T$ extending $K$, and $V$ be a geometrically irreducible variety defined over $K$.  For a quantifier-free $K$-formula $\phi(x)$, the following are equivalent:
\begin{description}
\item[(a)] $\phi(x)$ is $V$-dense.
\item[(b)] $\phi(L)$ is Zariski dense in $V(L^{alg})$.
\item[(c)] We can extend the $T_\forall$-structure on $K$ to the function field $K(V)$ in such a way that the generic point of $V$ in $K(V)$ satisfies $\phi(x)$.
\end{description}
\end{lemma}
\begin{proof}
\begin{description}
\item[(a) $\implies$ (b)]  Suppose $\phi(x)$ is $V$-dense.  Let $W$ be the Zariski closure of $\phi(L)$ in $V(L^{alg})$.  Then $W$ is defined over $L$ rather than $L^{alg}$, because $W$ is the Zariski closure of some $L$-points.  Therefore it makes sense to think of $W$ as a definable set.  If $\mathfrak{M}$ is a monster model of $T$ extending $L$, then $\dim V = \rk \phi(\mathfrak{M}) \le \rk W \le \dim W \le \dim V$.  Therefore $\dim W = \dim V$.  As $V$ is geometrically irreducible, $W = V$.
\item[(b) $\implies$ (a)] Let $\mathfrak{M}$ be a monster model of $T$ extending $L$, and let $n = \dim V$.  If $\phi(x)$ is not $V$-dense, then every element of $\phi(\mathfrak{M})$ has transcendence degree less than $n$ over $K$.  By compactness, $\phi(\mathfrak{M})$ is contained in a finite union of $K$-definable varieties of dimension less than $n$.  We may assume these varieties are closed subvarieties of $V$.  Of course $\phi(L)$ is also contained in this union, which is clearly a Zariski closed proper subset of $V$.  So $\phi(L)$ is not Zariski dense.
\item[(a) $\implies$ (c)] Embed $K$ into a monster model $\mathfrak{M}$.  Let $\alpha$ be a point in $\phi(\mathfrak{M}) \subseteq V(\mathfrak{M})$ with $\trdeg(\alpha/K) = \rk \phi(\mathfrak{M}) = \dim V$.  Then $\alpha$ is a generic point on $V$, i.e., $K(\alpha) \cong K(V)$.  And $\alpha$ satisfies $\phi(x)$.
\item[(c) $\implies$ (a)] Embed $K(V)$ into a monster model $\mathfrak{M}$.  Let $\alpha$ denote the generic point of $V$, so that $\mathfrak{M} \models \phi(\alpha)$ holds.  Clearly $\trdeg(\alpha/K) = \dim V$.  Thus $\rk_K \phi(\mathfrak{M}) \ge \trdeg(\alpha/K) = \dim V$, implying $V$-density of $\phi(x)$. \qedhere
\end{description}
\end{proof}

\begin{lemma}\label{ershov}
Let $L$ be a model of ACVF, and let $V \subseteq \mathbb{A}^n$ be an irreducible affine variety over $L$.  Suppose $0 \in V$.  Let $\mathcal{O}_L^n$ be the closed unit polydisk in $\mathbb{A}^n$.  Then $\mathcal{O}_L^n \cap V$ is Zariski dense in $V$.
\end{lemma}
This Lemma is essentially Lemma 1.1 in \cite{Ducros}, but we will give give a more elementary proof based on the proof of Proposition 4.2.1 in \cite{Ershov}.
\begin{proof}Let $L(\alpha)$ be the function field of $V$, obtained by adding a generic point $\alpha$ of $V$ to the field $L$.  By the implication (c) $\implies$ (b) of Lemma~\ref{tfae-variant} applied in the case where $\phi(x)$ is the formula defining $\mathcal{O}_L^n \cap V$,
it suffices to extend the valuation on $L$ to $L(\alpha)$ in such a way that every coordinate of $\alpha$ has nonnegative valuation.

Now $L[\alpha]$ is the coordinate ring of $V$, so the fact that $0 \in V$ implies that there is an $L$-algebra homomorphism $L[\alpha] \to L$ sending every coordinate of $\alpha$ to zero.  This yields an $\mathcal{O}_L$-algebra homomorphism $f : \mathcal{O}_L[\alpha] \to \mathcal{O}_L$ sending every coordinate of $\alpha$ to $0$.  Let $\mathfrak{m}$ be the maximal ideal of $\mathcal{O}_L$, and let $\mathfrak{p} = f^{-1}(\mathfrak{m})$.  Then $\mathfrak{p}$ is a prime ideal, and $\mathfrak{p} \cap \mathcal{O}_L = \mathfrak{m}$.  Also, as $f$ kills the coordinates of $\alpha$, the coordinates of $\alpha$ live in $\mathfrak{p}$.

Since $\mathcal{O}_L[\alpha]$ is a domain, there is a valuation $v'$ on $L(\alpha)$, the fraction field of $\mathcal{O}_L[\alpha]$, with the following properties:
\begin{itemize}
\item Every element of $\mathfrak{p}$ has positive valuation.  In particular, the elements of $\mathfrak{m}$ and the coordinates of $\alpha$ have positive valuation.
\item Every element of $\mathcal{O}_L[\alpha] \setminus \mathfrak{p}$ has valuation zero.  In particular, the elements of $\mathcal{O}_L^\times = \mathcal{O}_L \setminus \mathfrak{m}$ have valuation zero.
\end{itemize}
(Indeed, it is a general fact that if $S$ is a domain and $\mathfrak{p}$ is a prime ideal, then there is a valuation on the fraction field of $S$ which assigns a positive valuation to elements of $\mathfrak{p}$ and a vanishing valuation to elements of $S \setminus \mathfrak{p}$.  To find such a valuation, take a valuation ring in $\Frac(S)$ dominating the local ring $S_\mathfrak{p}$.)

The resulting valuation on $L(\alpha)$ extends the valuation on $L$, because it assigns positive valuation to elements in $\mathfrak{m}$, and zero valuation to elements in $\mathcal{O}_L \setminus \mathfrak{m}$.  Also, the valuation of any coordinate of $\alpha$ is positive, hence non-negative, so $\alpha$ lives in the closed unit polydisk.
\end{proof}

\begin{lemma}\label{tfae-smoothness}
Let $V$ be a geometrically irreducible affine variety over $K \models T_\forall$, and let $\phi(x)$ be a quantifier-free $K$-formula.  Let $L$ be a model of $T$ extending $K$.  Suppose $\phi(x)$ defines an open subset of $V(L)$.
\begin{description}
\item[(a)] If $T$ is ACVF, then $\phi(x)$ is $V$-dense if and only if $\phi(L)$ is non-empty.
\item[(b)] In general, $\phi(x)$ is $V$-dense if $\phi(L)$ contains a smooth point of $V$.
\end{description}
\end{lemma}
\begin{proof}
\begin{description}
\item[(a)] If $\phi(x)$ is $V$-dense, then certainly $\phi(L)$ is non-empty.  Conversely, suppose $\phi(L)$ is non-empty.  Let $p$ be a point in $\phi(L)$ and let $U$ be an open neighborhood of $p$, with $U \cap V \subseteq \phi(L)$.  There is some $L$-definable affine transformation $f$ which sends $p$ to the origin and moves $U$ so as to contain the closed unit polydisk.  Then $f(U \cap V) = f(U) \cap f(V)$ is Zariski dense in $f(V)$, by Lemma~\ref{ershov}.  So $\phi(L) \supseteq U \cap V$ is Zariski dense in $V$.  Thus $\phi(x)$ is $V$-dense, by Lemma~\ref{tfae-variant}.
\item[(b)] If $\phi(x)$ is $V$-dense, then $\phi(L)$ contains a smooth point of $V$, because the smooth locus of $V$ is a Zariski dense Zariski open.  Conversely, suppose $\phi(L)$ contains a smooth point $p$.
Note that $L$ is perfect.  The tangent space $T_p V$ is $L$-definable.  By Hilbert's Theorem 90, there is an $L$-definable basis of $T_p V$.  Therefore, after applying an $L$-definable change of coordinates, we may assume $T_p V$ is horizontal.  By the implicit function theorem, $V$ then looks locally around $p$ like the graph of a function.  In particular, the coordinate projection maps a neighborhood of $p$ homeomorphically to an open subset of affine $n$-space, where $n = \dim V$.  By Lemma~\ref{various}, this ensures that any neighborhood of $p$, such as $\phi(L)$, has rank at least $n$.  So $\phi(x)$ is $V$-dense. \qedhere
\end{description}
\end{proof}

\begin{remark}\label{hilbert90explanation}
  Here, and in Lemma~\ref{jacobian} below, we are using the
  \emph{model-theoretic} version of Hilbert's theorem 90.  This folk
  theorem says that if $M \models \ACF$, if $K$ is a perfect subfield
  of $M$, and if $V$ is a $K$-definable $M$-vector space, then $V$ admits
  a $K$-definable basis.  It is an easy exercise to derive this fact
  from standard Galois descent of vector spaces, which is part of
  Grothendieck's modern generalization of Hilbert's original Theorem
  90. (See III.4.10 in \cite{etalemilne}.)
\end{remark}

\begin{lemma}\label{whoops}
Let $V$ be a geometrically irreducible affine variety over $K \models
T_\forall$, and let $\phi(x)$ be a quantifier-free $K$-formula that is
$V$-dense.  Then there is a quantifier-free $K$-formula $\psi(x)$ that
is also $V$-dense, such that in any/every $L \models T$ extending $K$,
$\psi(L)$ is a topologically open subset of $V(L)$, and $\psi(L)
\subseteq \phi(L)$.
\end{lemma}
\begin{proof}
Choose some monster model $\mathfrak{M} \models T$ extending $K$ and let $\psi(\mathfrak{M})$ pick out the topological interior of $\phi(\mathfrak{M})$ inside $V(\mathfrak{M})$.  By quantifier-elimination, we can take $\psi(x)$ to be quantifier-free with parameters from $K$.  It remains to show that $\psi(\mathfrak{M})$ is $V$-dense.  Let $\alpha \in \phi(\mathfrak{M})$ have transcendence degree $n$ over $K$, where $n = \dim V$.  By Fact~\ref{the-fact}, $\phi(\mathfrak{M})$ can be written as a finite union of finite intersections of $K$-definable opens and varieties.  Let $X$ be one of these finite intersections, containing $\alpha$.  So $X = W \cap U$ for some $K$-variety $W$ and some $K$-definable open $U$.  As $\alpha \in W$ and $\alpha$ is a generic point on $V$, we must have $V \subseteq W$.  Then
\[ \alpha \in V \cap U \subseteq W \cap U \subseteq \phi(\mathfrak{M}).\]
But $V \cap U$ is a relative open in $V(\mathfrak{M})$, so it must be part of $\psi(\mathfrak{M})$.  In particular, $\alpha \in \psi(\mathfrak{M})$.  As $\trdeg(\alpha/K) = n$, we conclude that $\psi(x)$ is $V$-dense.
\end{proof}

\subsection{Forking and Dividing}
We continue to work in one of ACVF, RCF, or $p$CF.  Recall that RCF and $p$CF have definable Skolem functions in the home sort.  Thus if $S$ is a subset of the home sort, then $\acl(S) = \dcl(S)$ is a model.  In ACVF, $\acl(S)$ is the algebraic closure of $S$, which is a model unless $\acl(S)$ is trivially valued.

We will always be working in the home sort, rather than working with imaginaries.

\begin{lemma}\label{fork-div-1}
Let $S$ be a set (in the home sort) and let $\phi(x;b)$ be a formula.  Then $\phi(x;b)$ forks over $S$ if and only if it divides over $S$.
\end{lemma}
\begin{proof}
Indiscernibility over $S$ is the same thing as indiscernibility over $\acl(S)$, so $\phi(x;b)$ divides over $S$ if and only if it divides over $\acl(S)$.  Similarly, $\phi(x;b)$ forks over $S$ if and only if it forks over $\acl(S)$.  So we may assme $S = \acl(S)$.  If $T$ is RCF or $p$CF, then $S$ is a model, and therefore forking and dividing agree over $S$ by Theorem 1.1 of \cite{CK}.  If $T$ is ACVF, then forking and dividing agree over all sets, by Corollary 1.3 in \cite{CK}.
\end{proof}

We use $\forkindep$ to denote non-forking or non-dividing, and $\forkindep^{\ACF}$ to denote algebraic independence.

\begin{lemma}\label{stupid}
Let $\mathfrak{M}$ be a monster model of $T$, and let $B, C$ be small subsets of $\mathfrak{M}$, with $B$ finite. Then we can find a sequence $B_0, B_1, B_2, \ldots$ in $\mathfrak{M}$ that is $C$-indiscernible, such that $B_0 = B$ and $B_i \forkindep_C^{\ACF} B_{< i}$ for every $i$.
\end{lemma}
\begin{proof}
We may assume that $B$ is ordered as a tuple in such a way that the first $k$ elements of $B$ are a transcendence basis of $B$ over $C$.
Construct a sequence $D_0, D_1, \ldots$ of realizations of $\tp(B/C)$ such that $D_i \forkindep_C^{\ACF} D_{< i}$ for every $i$.  This is possible by using Corollary~\ref{extend-types-independently} to extend $\tp(B/C)$ to a type over $CD_{< i}$ having the same transcendence degree over $CD_{< i}$ as over $C$.
Let $B_0, B_1, B_2, \ldots$ be a $C$-indiscernible sequence modeled on $D_0, D_1, \ldots$.  Let $\pi(X)$ pick out the first $k$ elements of a tuple $X$.  Then $\pi(D_0)\concatA \pi(D_1)\concatA \pi(D_2)\concatA \cdots$ is an algebraically independent sequence of singletons over $C$.  This is part of the EM-type of the $D_i$ over $C$, so it is also true that $\pi(B_0)\concatA \pi(B_1)\concatA \pi(B_2)\concatA \cdots$ is an algebraically independent sequence of singletons over $C$.  Since $D_i \equiv_C B$ for every $i$, we also have $B_i \equiv_C B$ for every $i$.  Thus $\pi(B_i)$ is a transcendence basis for $B_i$ over $C$, and we conclude that $B_i \forkindep^{\ACF}_C B_{< i}$ for every $i$.  Finally, moving the $B_i$ by an automorphism over $C$, we may assume that $B_0 = B$.
\end{proof}

\begin{lemma}\label{stronger-than-acf}
$A \forkindep_C B$ implies $A \forkindep^{\ACF}_C B$.
\end{lemma}
\begin{proof}
  Assume $A \forkindep_C B$.  By Lemma~\ref{stupid}, we can find a
  sequence $B_0, B_1, B_2, \ldots$ of realizations of $\tp(B/C)$,
  indiscernible over $C$, and satisfying $B_i \forkindep_C^{\ACF}
  B_{<i}$ for every $i$.  Suppose for the sake of contradiction that
  in some ambient model of ACF, $\tp(A/BC)$ contains a formula
  $\phi(X;Y)$ which divides (in the ACF sense) over $C$.  By
  quantifier elimination in ACF, we may assume $\phi$ is
  quantifier-free.  In stable theories such as ACF, dividing is
  witnessed in any Morley sequence.  In particular
  \[ \bigwedge_i \phi(X;B_i)\]
  is inconsistent in the ambient model of ACF, hence inconsistent in
  the original smaller structure.  Thus $\phi(X;B)$ forks and divides
  over $C$ in the original structure, a contradiction.
\end{proof}




Lastly, we show that dividing is always witnessed by an algebraically independent sequence.
\begin{lemma}\label{dubious-ist}
If a formula $\phi(x;a)$ divides over a set $A$, then the dividing is witnessed by an $A$-indiscernible sequence $a = a_0, a_1, a_2, \ldots$ such that $a_i \forkindep^{\ACF}_A a_{< i}$ for every $i$.
\end{lemma}
\begin{proof}
Apply Claim 3.10 of \cite{CK} with the abstract independence relation taken to be $\forkindep$ (non-forking).  Forking satisfies (1)-(7) of \cite{CK} Definition 2.9 by Fact 2.12(5) of \cite{CK}.  And $A$ is an extension base for forking by Lemma~\ref{fork-div-1} above and Theorem 1.1 of \cite{CK} (or by Fact 2.14 of \cite{CK} in the cases other than $p$CF).  So Claim 3.10 of \cite{CK} is applicable.  Consequently we get a model $M$ containing $A$, a global type $p$ extending $\tp(a/M)$, $\forkindep$-free over $A$, such that any/every Morley sequence generated by $p$ over $M$ witnesses the dividing of $\phi(x;a)$.  Because $\forkindep$ is stronger than Lascar invariance, any such Morley sequence will be $M$-indiscernible, hence $A$-indiscernible.  Because $\forkindep$ is stronger than algebraic independence (Lemma~\ref{stronger-than-acf}), and $p$ is $\forkindep$-free over $A$, any Morley sequence $a_0, a_1, \ldots$ generated by $p$ will be algebraically independent over $A$.  Specifically, $a_i \models p|_{M a_{< i}}$, so as $p$ is $\forkindep$-free over $A$, $a_i \forkindep_A M a_{< i}$, and hence $a_i \forkindep_A^{\ACF} a_{< i}$.
\end{proof}

\section{The Model Companion}\label{review}
Now we turn our attention to fields with several valuations, several orderings, and several $p$-valuations.  For $1 \le i \le n$, let $T_i$ be one of ACVF, RCF, or $p$CF (in the same languages as in the previous section).  Let $\mathcal{L}_i$ denote the language of $T_i$; assume that $\mathcal{L}_i \cap \mathcal{L}_j = \mathcal{L}_{rings}$ for $i \ne j$.  Let $T^0$ be $\bigcup_{i = 1}^n (T_i)_\forall$, the theory that would be denoted $((T_1)_\forall, (T_2)_\forall, \ldots, (T_n)_\forall)$ in van den Dries's notation.  Technically speaking, models of $T^0$ should be allowed to be domains, rather than fields.  However, we will assume that $T^0$ also includes the field axioms, sweeping domains under the rug.

One essentially knows that $T^0$ has a model companion $T$ by Chapter
III of van den Dries's thesis \cite{LvdD}.  We will quickly reprove
the existence of $T$ in this section, expressing the axioms of the
model companion in a more geometric and less syntactic form, and also
including the case of positive characteristic explicitly.

\subsection{The Axioms}\label{axiom-section}

Consider the following axioms that a model $K$ of $T^0$ could satisfy:
\begin{description}
\item[A1:] $K$ is existentially closed with respect to finite extensions, i.e., if $L/K$ is a finite algebraic extension and $L \models T^0$, then $L = K$.
\item[A1':] For every irreducible polynomial $P(X) \in K[X]$ of degree greater than 1, there is some $1 \le i \le n$ such that $P(x) = 0$ has no solution in any/every model of $T_i$ extending $K \restriction \mathcal{L}_i$.
\item[A2(m):] Let $V$ be an $m$-dimensional geometrically irreducible variety over $K$.  For $1 \le i \le n$, let $\phi_i(x)$ be a $V$-dense quantifier-free $\mathcal{L}_i$-formula with parameters from $K$.  Then $\bigcap_{i =1}^n \phi_i(K) \ne \emptyset$.
\item[A2($\le m$):] A2($m'$) holds for all $m' \le m$.
\item[A2:] A2(m) holds, for all $m$
\end{description}

\begin{remark}\label{a1-restatement}
For $K \models T^0$, A1 and A1' are equivalent.
\end{remark}
\begin{proof}
Suppose $K$ satisfies A1, and $P(X) \in K[X]$ is irreducible of degree greater than 1.  Suppose that for every $1 \le i \le n$, there is a solution $\alpha_i$ of $P(x) = 0$ in a model $M_i \models T_i$ extending $K \restriction \mathcal{L}_i$.  Then we can extend the $\mathcal{L}_i$-structure from $K$ to $K(\alpha) \cong K[X]/P(X)$.  Because this holds for every $i$, we can endow $K[X]/P(X)$ with the structure of a model of $T^0$.  By A1, $K[X]/P(X)$ must be $K$, so $P(X)$ has degree 1.

Conversely, suppose $K$ satisfies A1' but not A1.  Let $L/K$ be a counterexample to A1, and take some $\alpha \in L \setminus K$.  Let $P(X)$ be the irreducible polynomial of $\alpha$ over $K$.  This polynomial must have degree greater than 1.  For each $i$ let $M_i$ be a model of $T_i$ extending $L \restriction \mathcal{L}_i$.  Then $P(x) = 0$ has a solution in $L$, hence in $M_i$, which is a model of $T_i$ extending $K$.  This contradicts A1'.
\end{proof}

\begin{lemma}\label{main-lemma}
Let $K$ be a model of $T^0$, and $m \ge 1$.  The following are equivalent:
\begin{description}
\item[(a)] For every model $L$ of $T^0$ extending $K$, for every tuple $\alpha$ from $L$ with $\trdeg(\alpha/K) \le m$, the quantifier-free type $\qftp(\alpha/K)$ is finitely satisfiable in $K$.
\item[(b)] $K$ satisfies A1 and A2($\le m$).
\end{description}
\end{lemma}
\begin{proof}
(a) $\implies$ (b)
For A1, suppose that $L/K$ is a finite extension, and $L \models T$.  If $\alpha \in L$, then $\alpha$ is algebraic over $K$, so $\trdeg(\alpha/K) = 0 \le m$.  By (a), the quantifier-free type of $\alpha$ is realized in $K$.  So the irreducible polynomial of $\alpha$ over $K$ has a zero in $K$, implying $\alpha \in K$.  As $\alpha \in L$ was arbitrary, $L = K$.

For A2($m'$), let $V$ be an $m'$-dimensional geometrically irreducible variety over $K$.  For $1 \le i \le n$, let $\phi_i(x)$ be a $V$-dense quantifier-free $\mathcal{L}_i$-formula with parameters from $K$.  By Lemma \ref{tfae-variant}(c), we can extend the $\mathcal{L}_i$-structure to $K(V)$ in such a way that the generic point satisfies $\phi_i(x)$.  Doing this for all $i$, we make $K(V)$ be a model of $T^0$ extending $K$, such that if $\alpha \in K(V)$ denotes the generic point, then $\bigwedge_{i = 1}^n \phi_i(\alpha)$ holds.  Now $\trdeg(\alpha/K) = \dim V \le m$, so by (a), $\qftp(\alpha/K)$ is finitely satisfiable in $K$.  In particular, the formula $\bigwedge_{i = 1}^n \phi(x)$ is satisfiable in $K$, which is the conclusion of A2($m'$).

(b) $\implies$ (a).  Suppose $L$ is a model of $T^0$ extending $K$ and
$\alpha$ is a tuple from $L$, with $\trdeg(\alpha/K) \le m$.  By A1,
$K$ is relatively algebraically closed in $L$.  Let $V$ be the
$K$-variety of which $\alpha$ is a generic point.  Then $V$ is
geometrically irreducible, by Fact~\ref{cb-stationarity-variety}.
Also, $m' := \dim V = \trdeg(\alpha/K) \le m$.  Let $\psi(x)$ be any
formula in $\qftp(\alpha/K)$.  We want to show that $\psi$ is
satisfied by an element of $K$.  We may assume that $\psi(x)$ includes
the statement that $x \in V$.  By Fact~\ref{the-fact}, $\psi(x)$ is a
positive boolean combination of statements of the form
\begin{itemize}
\item $x \in W$, for some $K$-definable affine variety $W$.  Since we intersected $\psi(x)$ with $V$, we may assume $W \subseteq V$.
\item $\theta(x)$, where $\theta(x)$ is a quantifier-free $\mathcal{L}_i$-formula for some $i$, such that $\theta(L)$ is an open subset of the ambient affine space, for any/every $L \models T_i$ extending $K \restriction \mathcal{L}_i$.
\end{itemize}
Writing $\psi(x)$ as a disjunction of conjunctions of such statements, and replacing $\psi(x)$ by whichever disjunct $\alpha$ satisfies, we may assume that $\psi(x)$ is a conjunction of such statements.  An intersection of $K$-varieties is a $K$-variety, and an intersection of open subsets of affine space is an open subset of affine space, so we may assume
\[ \psi(x) \equiv ``x \in W\text{''} \wedge \bigwedge_{i = 1}^n \phi_i(x),\]
where $W$ is some $K$-variety contained in $V$, and where $\phi_i(x)$ is a quantifier-free $\mathcal{L}_i$-formula defining an open subset of the ambient affine space, when interpreted in any/every model of $T_i$ extending $K \restriction \mathcal{L}_i$.

Because $\alpha$ satisfies $\psi(x)$, and $\alpha$ is a generic point of $V$, $W$ must be $V$.  Rewrite $\psi$ as $\bigwedge_{i = 1}^n \phi'_i(x)$, where each $\phi'_i(x)$ asserts that $x \in V$ and $\phi_i(x)$ holds.  Because $K$ satisfies axiom A2($m'$), $\psi(x)$ will be satisfiable in $K$ as long as $\phi'_i(x)$ is $V$-dense for each $i$.  But note that $L$ provides a way of extending the $\mathcal{L}_i$-structure from $K$ to $K(\alpha) \cong K(V)$ in such a way that $\phi'_i(\alpha)$ holds, so $\phi'_i$ is $V$-dense by Lemma~\ref{tfae-variant}(c).
\end{proof}

\begin{theorem} \label{axioms-are-correct}
The theory $T^0$ has a model companion $T$, whose models are exactly the $K \models T^0$ satisfying A1 and A2.
\end{theorem}
\begin{proof}
It is well known that a model $K$ is existentially closed if and only
if for every model $L$ extending $K$ and for every tuple $\alpha$ from
$L$, the quantifier-free type $\qftp(\alpha/K)$ is finitely
satisfiable in $K$.  So by Lemma~\ref{main-lemma}, a model of $T^0$ is
existentially closed if and only if it satisfies A1 and A2.  By basic
facts about model companions of $\forall\exists$-theories, it remains
to show that A1 and A2 are first order.  For A1, this comes from
Remark~\ref{a1-restatement}, because A1' is first order by
quantifier-elimination in the $T_i$.  Axiom A2 is first order by
quantifier-elimination in the $T_i$, by Lemma~\ref{various}(c), and by
the fact that geometric irreducibility is definable by a
quantifier-free formula in the language of fields (this is well-known
and proven in Chapter IV of \cite{LvdD}).
\end{proof}

Henceforth, we will use $T$ to denote the model companion.  Also, we
will write $T_\forall$ for $T^0$, sweeping the distinction
between domains and fields under the rug.

We make several remarks about the axioms:
\begin{remark}\label{a1-simplification}
In the case where $T_i$ is ACVF for $i > 1$, axiom A1 merely says that $K \restriction \mathcal{L}_1$ is a model of $T_1$, i.e., is algebraically closed or real closed or $p$-adically closed.
\end{remark}
\begin{remark}\label{smooth-hmmm}
In Axiom A2($m$), it suffices to consider the case of smooth $V$.  If $V$ is \emph{not} smooth, one can find an open subvariety $V'$ of $V$ which is smooth, and which is isomorphic to an affine variety.  (Use the facts that the smooth locus of an irreducible variety is a Zariski dense Zariski open, and that the affine open subsets of a scheme form a basis for its topology.)  If $\phi_i(x)$ is $V$-dense, then $\phi_i(x) \wedge \text{``}x \in V'\text{''}$ is $V'$-dense, essentially by Lemma~\ref{tfae-variant}(b). Then applying the smooth case of A2($m$) to $V'$ yields a point in $V'$ satisfying $\bigwedge_{i = 1}^n \phi_i(x)$.
\end{remark}
\begin{remark}\label{open-whoops}
In Axiom A2, it suffices to consider $V$-dense formulas $\phi_i(x)$ such that $\phi_i(L)$ defines an \emph{open} subset of $V(L)$ for any/every $L \models T$ extending $K \restriction \mathcal{L}_i$.  This follows by Lemma~\ref{whoops}.
\end{remark}
\begin{remark}\label{combo-restatement}
We can combine the previous two remarks.  Then Lemma~\ref{tfae-smoothness}(b), yields the following restatement of A2($m$): if $V$ is a geometrically irreducible $m$-dimensional smooth affine variety defined over $K$, and if $\phi_i(x)$ is a quantifier-free $\mathcal{L}_i$-formula over $K$ for each $1 \le i \le n$, and if $\phi_i(K_i)$ is a \emph{non-empty} open subset of $V(K_i)$ for any/every $K_i \models T$ extending $K \restriction \mathcal{L}_i$, then $\bigcap_{i = 1}^n \phi_i(K) \ne \emptyset$.
\end{remark}
\begin{remark}\label{a2-simplification}
If every $T_i$ is ACVF, then A1 merely says that $K$ is algebraically closed.  Consequently, in Remark~\ref{combo-restatement} the $K_i$ can be taken to be $K$ itself.  Thus A2(m) ends up being equivalent to the statement that if $V$ is a smooth irreducible $m$-dimensional affine variety, and $\phi_i(x)$ is a quantifier-free $\mathcal{L}_i$-formula defining a non-empty open subset of $V$ for $1 \le i \le n$, then $\bigcap_{i = 1}^n \phi_i(K)$ is non-empty.  Even more concisely, this means that for every smooth $m$-dimensional variety $V$, the diagonal map $V(K) \to \prod_{i = 1}^n V(K)$ has dense image in the product topology, using the topology from the $i^{\text{th}}$ valuation for the $i^{\text{th}}$ entry in the product.

In fact, in Section~\ref{sec-miracle}, we will see that it suffices to check the case of $V = \mathbb{A}^1$, the affine line(!)
\end{remark}

\subsection{Quantifier-Elimination up to Algebraic Covers}
As in the previous section, $T_\forall$ is the theory of fields with $(T_i)_\forall$ structure for each $1 \le i \le n$, and $T$ is the model companion of $T_\forall$.

\begin{lemma}\label{precise-amalgamation}
Let $K$ be a model of $T_\forall$.  Let $L$ and $L'$ be two models of $T_\forall$ extending $K$.  Suppose that $K$ is relatively separably closed in $L$ or $L'$ (Definition~\ref{nearly}).  Then $L$ and $L'$ can be amalgamated over $K$, and this can be done in such a way that $L$ and $L'$ are algebraically independent over $K$.
\end{lemma}
\begin{proof}
For each $1 \le i \le n$, we can find some amalgam $M_i \models (T_i)_\forall$ of $L \restriction \mathcal{L}_i$ and $L' \restriction \mathcal{L}_i$ over $K \restriction \mathcal{L}_i$, by Corollary~\ref{free-amalgamation}.  The resulting compositums $LL'$ must be isomorphic on the level of fields, by Fact~\ref{cb-stationarity-fact}.  Consequently, we can endow the canonical field $LL'$ with a $(T_i)_\forall$-structure extending those on $L$ and $L'$, for each $i$.  This gives $LL'$ the structure of a $T_\forall$-model.  And $L$ and $L'$ are algebraically independent inside $LL'$.
\end{proof}

\begin{corollary}\label{a1-rel-closed}
Let $K$ be a model of $T_\forall$ and let $L$ be a model of $T$ extending $K$.  Then $K$ is relatively algebraically closed in $L$ if and only if $K$ satisfies axiom A1.  (In particular, this does not depend on $L$.)
\end{corollary}
\begin{proof}
If $K$ satisfies axiom A1, then obviously $K$ is relatively
algebraically closed in $L$.  Conversely, suppose that $K$ is
relatively algebraically closed in $L$ but does not satisfy A1.  Then
there is some model $L'$ of $T_\forall$ extending $K$, with $L'/K$
finite and $L' \ne K$.  By Lemma~\ref{precise-amalgamation}, we can
amalgamate $L$ and $L'$ over $K$.  Embed the resulting compositum
$LL'$ in a model $M$ of $T$.  Because $T$ is model-complete, $L
\preceq M$.  Now choose some $\alpha \in L' \setminus K$.  The
irreducible polynomial of $\alpha$ over $K$ has a root in $M$, and
hence has a root in $L$, contradicting the fact that $K$ is relatively
algebraically closed in $L$.
\end{proof}

\begin{corollary}\label{qe-version-1}
Let $K$ be model of $T_\forall$, and suppose $K$ satisfies A1.  Then the type of $K$ is determined, i.e., if $L$ and $L'$ are two models of $T$ extending $K$, then $K$ has the same type in $L$ and $L'$.  Equivalently, the diagram of $K$ implies the elementary diagram of $K$, modulo the axioms of $T$.
\end{corollary}
\begin{proof}
By Corollary~\ref{a1-rel-closed}, $K$ is relatively algebraically closed in $L$ and $L'$.  So we can amalgamate $L$ and $L'$ over $K$, by Lemma~\ref{precise-amalgamation}.  If $M$ is a model of $T$ extending $LL'$, then by model completeness $L \preceq M \succeq L'$, ensuring that $K$ has the same type in each.
\end{proof}

\begin{corollary}\label{acl-char}
In models of $T$, field-theoretic algebraic closure agrees with model-theoretic algebraic closure.
\end{corollary}
\begin{proof}
Let $M$ be a model of $T$.  Let $S$ be a subset of $M$.  Let $K$ be the field-theoretic algebraic closure of $S$, i.e., the relative algebraic closure of $S$ in $M$.  By Lemma~\ref{precise-amalgamation}, we can amalgamate $M$ and a copy $M'$ of $M$ over $K$ in a way that makes $M$ and $M'$ be algebraically independent over $K$.    Embedding $MM'$ into a model $N$ of $T$, and using model completeness, we get $M \preceq N \succeq M'$.  Now $\acl(S)$ is the same when computed in $M$, $N$, or $M'$.  In particular, $\acl(S) \subseteq M \cap M'$.  Since $M$ and $M'$ are algebraically independent over $K$ and $K$ is relatively algebraically closed in each, $M \cap M' = K$.  Thus $\acl(S) \subseteq K$.  Obviously $K \subseteq \acl(S)$.
\end{proof}

For $K$ a field, let $\Abs(K)$ denote the algebraic closure of the prime field in $K$.
\begin{corollary}\label{completions}
Two models $M_1, M_2 \models T$ are elementarily equivalent if and only if $\Abs(M_1)$ and $\Abs(M_2)$ are isomorphic as models of $T_\forall$.
\end{corollary}
\begin{proof}
If $M_1$ and $M_2$ are elementarily equivalent, we can embed them as elementary substructures into a third model $M_3 \models T$.  Then $\Abs(M_1) = \Abs(M_3) = \Abs(M_2)$, so certainly $\Abs(M_1)$ is isomorphic to $\Abs(M_2)$.

Conversely, suppose $\Abs(M_1) \cong \Abs(M_2)$.  Then, as $\Abs(M_1)$ is relatively algebraically closed in $M_1$ and in $M_2$, it follows by Corollaries \ref{a1-rel-closed} and \ref{qe-version-1} that we can amalgamate $M_1$ and $M_2$ over $\Abs(M_1)$.  Embedding the resulting compositum into a model of $T$ and using model completeness, we get $M_1 \equiv M_2$.
\end{proof}

\begin{corollary}\label{qe-version-special}
Suppose $T_1 \ne \ACVF$ and $T_i$ \emph{is} ACVF for $i > 1$.  Consider the expanded theory where we add in symbols for every zero-definable $T_1$-definable function.  (This makes sense because if $M \models T$, then $M \restriction \mathcal{L}_1 \models T_1$, by Remark \ref{a1-simplification}.)  Then $T$ has quantifier-elimination.
\end{corollary}
\begin{proof}
After adding in these new symbols, a substructure is the same as a subfield $K$ closed under all $T_1$-definable functions.  As RCF and $p$CF have definable Skolem functions, this is equivalent to $K \restriction \mathcal{L}_1$ being a model of $T_1$, which is equivalent to $K$ satisfying axiom A1, as noted in Remark \ref{a1-simplification}.  Now apply Corollary~\ref{qe-version-1} to get substructure completeness, which is the same thing as quantifier-elimination.
\end{proof}
This probably also holds if $T_i \ne \ACVF$ for more than one $i$,
though the extra functions would become partial functions.

Without adding in extra symbols, quantifier elimination fails.  But we
still get quantifier-elimination up to algebraic covers, in a certain
sense.
\begin{theorem}\label{qe-version-2}
In $T$, every formula $\phi(\vec{x})$ is equivalent to one of the form
\begin{equation}
  \exists y : \left(P(y,\vec{x}) = 0 \wedge \psi(y,\vec{x})\right), \label{form}
\end{equation}
where $y$ is a singleton, $\psi(y,\vec{x})$ is quantifier-free, and $P(y,x)$ is a polynomial in $\Zz[\vec{x},y]$, monic as a polynomial in $y$.
\end{theorem}
\begin{proof}
Let $\Sigma(\vec{x})$ be the set of all formulas of the form (\ref{form}).  First we observe that $\Sigma(\vec{x})$ is closed under disjunction, because
\[ \left( \exists y : P(y,\vec{x}) = 0 \wedge \psi(y,\vec{x}) \right) \vee \left( \exists y : Q(y,\vec{x}) = 0 \wedge \psi'(y,\vec{x}) \right)\]
is equivalent to
\[ \exists y : P(y,\vec{x})Q(y,\vec{x}) = 0 \wedge \psi''(y,\vec{x}),\]
where $\psi''(y,\vec{x})$ is the quantifier-free formula
\[ \left(P(y,\vec{x}) = 0 \wedge \psi(y,\vec{x}) \right) \vee \left(Q(y,\vec{x}) = 0 \wedge \psi'(y,\vec{x}) \right).\]

Now given a formula $\phi(\vec{x})$, not quantifier-free, let $\Sigma_0(\vec{x})$ be the set of formulas in $\Sigma(\vec{x})$ which imply $\phi(\vec{x})$, i.e.,
\[ \Sigma_0(\vec{x}) = \{ \sigma(\vec{x}) \in \Sigma(\vec{x}) ~|~ T \vdash \forall \vec{x} : ( \sigma(\vec{x}) \rightarrow \phi(\vec{x}))\}.\]
Of couse $\Sigma_0(\vec{x})$ is closed under disjunction.
It suffices to show that $\phi(\vec{x})$ implies a finite disjunction of formulas in $\Sigma_0(\vec{x})$, because then $\phi(\vec{x})$ implies and is implied by a formula in $\Sigma_0(\vec{x})$.

Suppose for the sake of contradiction that $\phi(\vec{x})$ does not imply a finite disjunction of formulas in $\Sigma_0(\vec{x})$.  Then the partial type
\[ \{\phi(\vec{x})\} \cup \{\neg \sigma(\vec{x}) : \sigma(\vec{x}) \in \Sigma_0(\vec{x})\}\]
is consistent with $T$.  Let $M$ be a model of $T$ containing a tuple $\vec{\alpha}$ realizing this partial type.
So $\phi(\vec{\alpha})$ holds in $M$, but not because of any formula of the form (\ref{form}).

Let $R$ be the ring $\mathbb{Z}[\vec{\alpha}] \subseteq M$.  Let $K \subseteq M$ be the smallest perfect field containing $R$; note that $M$ itself is perfect so this makes sense.  Indeed, if every $T_i$ is ACVF, then $M$ is algebraically closed by Remark~\ref{a1-simplification}.  Otherwise, one of the $T_i$'s is RCF or $p$CF, making $M$ be characteristic zero.

Let $\overline{K}$ be the relative algebraic closure of $K$ (or equivalently, $\vec{\alpha}$) inside $M$.  By Corollaries \ref{a1-rel-closed} and \ref{qe-version-1}, the diagram of $\overline{K}$ implies the elementary diagram of $\overline{K}$.  In particular, the diagram of $\overline{K}$ implies $\phi(\vec{\alpha})$.  By compactness, the diagram of $L$ implies $\phi(\vec{\alpha})$, for some finite extension $L$ of $K$.  Because $K$ is perfect, $L = K(\beta)$ for some \emph{singleton} $\beta$.  Multiplying $\beta$ by an appropriate element from $R$, we may assume that $\beta$ is integral over $R$.  Note that $L$ is perfect, because it is an algebraic extension of a perfect field, and in fact $L$ is the smallest perfect field containing $\vec{\alpha}$ and $\beta$.

As the diagram of $L$ implies $\phi(\vec{\alpha})$, so does the diagram of $\mathbb{Z}[\vec{\alpha},\beta]$, by Lemma~\ref{confusing2} below.
By compactness, there is some quantifier-free formula $\psi(y,\vec{x})$ which is true of $(\beta,\vec{\alpha})$ such that
\[ T \vdash \forall y ~ \forall \vec{x} : \psi(y,\vec{x}) \rightarrow \phi(\vec{x}).\]
Let $P(y,\vec{x})$ be the polynomial witnessing integrality of $\beta$ over $R$.  Then clearly
\[ T \vdash \forall \vec{x} : \left(\exists y :  P(y,\vec{x}) = 0 \wedge \psi(y,\vec{x}) \right) \rightarrow \phi(\vec{x}),\]
so $\exists y :  P(y,\vec{x}) = 0 \wedge \psi(y,\vec{x})$ is in $\Sigma_0(\vec{x})$, contradicting the fact that it holds of $\vec{\alpha}$ in $M$.
\end{proof}

\begin{lemma}\label{confusing2}
Let $M$ be a model of $T$ and $R$ be a subring of $M$.  Let $K \subseteq M$ be the smallest perfect field containing $R$.  Let $\alpha$ be a tuple from $R$, and $\phi(x)$ be a formula such that $M \models \phi(\alpha)$.  If $T$ and the diagram of $K$ imply $\phi(\alpha)$, then $T$ and the diagram of $R$ imply $\phi(\alpha)$.
\end{lemma}
\begin{proof}
If not, then there is a model $N$ of $T$ extending $R$, in which $\phi(\alpha)$ fails to hold.  This model $N$ must not satisfy the diagram of $K$.  Now $N$ certainly contains a copy of the pure field $K$, because the fraction field and perfect closure of a domain are unique.  Consequently, there must be at least two ways to extend the $T$-structure from $R$ to $K$, one coming from $M$ and one coming from $N$.  But this is absurd, because each valuation/ordering/$p$-valuation on $R$ extends uniquely to $K$, by quantifier elimination in the $T_i$.
\end{proof}

\subsection{Simplifying the axioms down to curves}

\begin{lemma}\label{curve-reduction-temp}
Let $K$ be an $\aleph_1$-saturated and $\aleph_1$-strongly homogeneous model of $T_\forall$ satisfying axioms A1 and A2(1).  Let $\mathfrak{M}$ be a monster model of $T$ extending $K$.  Let $S$ be a countable subset of $K$ and $\alpha$ be a countable tuple from $\mathfrak{M}$.  Then $\tp(\alpha/S)$ is realized in $K$.
\end{lemma}
\begin{proof}
Consider the following statements:
\begin{itemize}
\item $A_k$: if $\alpha$ is a finite tuple from $\mathfrak{M}$, with $\trdeg(\alpha/S) \le k$, then $\qftp(\alpha/S)$ is realized in $K$.
\item $B_k$: if $\alpha$ is a countable tuple from $\mathfrak{M}$, with $\trdeg(\alpha/S) \le k$, then $\qftp(\alpha/S)$ is realized in $K$.
\item $C_k$: if $\alpha$ is a countable tuple from $\mathfrak{M}$, with $\trdeg(\alpha/S) \le k$, then $\tp(\alpha/S)$ is realized in $K$.
\end{itemize}
There are several implications between these statements:
\begin{itemize}
\item For each $k$, $A_k$ implies $B_k$, by compactness.
\item For each $k$, $B_k$ implies $C_k$.  Indeed, if $\alpha$ is as in $C_k$, apply $B_k$ to $\alpha' := \acl(\alpha S)$ and use Corollary~\ref{qe-version-1}.
\item $C_k$ for all $k$ implies the statement of the Lemma, by compactness.
\end{itemize}
Finally, observe that $C_k$ and $C_j$ imply $C_{k+j}$: if $\alpha$ has transcendence degree $k + j$ over $S$, let $\beta$ be a subtuple of $\alpha$ with transcendence
degree $k$.  Then $\trdeg(\beta/S) \le k$ and $\trdeg(\alpha/\beta S) \le j$.  By $C_k$, we can apply an automorphism over $S$ to move $\beta$ inside $K$.  By $C_j$ applied to $\tp(\alpha/\beta S)$, we can then find a further automorphism moving $\alpha$ inside $K$.

Lemma~\ref{main-lemma} and $\aleph_1$-saturation of $K$ imply $A_1$.  By the above comments, this implies $C_1$, which in turn implies $C_{1+1}, C_3, C_4, \ldots$.  By compactness, the Lemma is true.
\end{proof}

\begin{theorem}\label{curve-reduction}
A field $K \models T_\forall$ is existentially closed, i.e., a model of $T$, if and only if it satisfies A1 and A2(1).
\end{theorem}
\begin{proof}
If $K$ is existentially closed, then certainly $K$ satisfies A1 and A2(1).  Conversely, suppose $K$ satisfies A1 and A2(1).  Let $K'$ be an $\aleph_1$-saturated $\aleph_1$-strongly homogeneous elementary extension of $K$.  As $K \equiv K'$, it suffices to show that $K' \models T$.  Let $\mathfrak{M}$ be a monster model of $T$, extending $K'$.  It suffices to show that $K' \preceq \mathfrak{M}$.  It suffices to show that if $D$ is a non-empty $K'$-definable subset of $\mathfrak{M}$, then $D$ intersects $K'$.  Let $S$ be a finite subset of $K'$ that $D$ is defined over, and let $\alpha$ be a point in $D$.  By Lemma~\ref{curve-reduction-temp}, $\tp(\alpha/S)$ is realized in $K'$.  Such a realization must live in $D$.
\end{proof}

Consequently, in checking the axioms one only needs to consider curves.  In fact, one only needs to consider smooth curves, by Remark~\ref{smooth-hmmm}.

\section{A Special Case}\label{sec-miracle}
In the case where almost every $T_i$ is ACVF, the axioms can be drastically simplified.

\begin{theorem}\label{miracle}
Suppose $T_2, \ldots T_n$ are all ACVF.  A model $K \models T_\forall$ is existentially closed (i.e., a model of $T$) if and only if the following three conditions hold:
\begin{itemize}
\item $K \restriction \mathcal{L}_1 \models T_1$
\item Each valuation $v_2, \ldots, v_n$ is non-trivial.
\item $T_i$ and $T_j$ do not induce the same topology on $K$, for $i \ne j$.
\end{itemize}
\end{theorem}
For example, if we are considering the theory of ordered valued
fields, this says that a model is existentially closed if and only if
the field is real closed, the valuation is non-trivial, and the
ordering and valuation induce different topologies on $K$.  A field
with $n$ valuations is existentially closed if and only if it is
algebraically closed and the valuations induce distinct non-discrete
topologies on the field.  Using this, we can easily see that
$\mathbb{Q}^{alg}$ with $n$ distinct valuations is an existentially
closed field with $n$ valuations.  This surprised me, since I expected
the Rumely Local-Global principle (Theorem 1 of \cite{rumely}) to be
necessary in the proof.

Theorem~\ref{miracle} is \emph{not} model theoretic, and is presumably
known to experts in algebraic geometry or field theory.  

In the proof of Theorem~\ref{miracle}, we will use A. L. Stone's Approximation Theorem (\cite{stoneapprox}, Theorem 3.4):
\begin{fact}\label{strong-approximation}
Let $K$ be a field.  Let $t_1, \ldots, t_n$ be topologies on $K$ arising from orderings and non-trivial valuations.  Suppose that $t_i \ne t_j$ for $i \ne j$.  Then the $\{t_i\}$ are independent, i.e., if $U_i$ is a non-empty $t_i$-open subset of $K$ for each $i$, then $\bigcap_{i = 1}^n U_i$ is non-empty.  Equivalently, the diagonal map $K \to \prod_{i = 1}^n K$ has dense image with respect to the product topology, using the topology $t_i$ for the $i^{\text{th}}$ term in the product.
\end{fact}
Note that Fact~\ref{strong-approximation} does not contradict the
existence of valuations which refine each other, because two
non-trivial valuations which refine each other always induce the same
topology.  A self-contained model-theoretic proof of Stone Approximation is
given in \cite{prestel-ziegler}, Theorem 4.1.

Also, we will need the following straightforward lemma.
\begin{lemma}\label{jacobian}
Let $K$ be a model of $T$.  Let $C$ be an affine smooth curve over $K$, geometrically irreducible.  Let $\overline{C}$ be the canonical smooth projective model (as an abstract variety).  For each $i$, let $\phi_i(x)$ be a $C$-dense quantifier-free $\mathcal{L}_i$-formula with parameters from $K$.  Then we can find a $K$-definable rational function $f : \overline{C} \to \mathbb{P}^1$ which is non-constant, and has the property that the divisor $f^{-1}(0)$ is a sum of distinct points in $\bigcap_{i = 1}^n \phi_i(K)$, with no multipliticities.  (In particular, the support of the divisor contains no points from $C(K^{alg}) \setminus C(K)$ and no points from $\overline{C} \setminus C$.)
\end{lemma}
\begin{proof}
Let $g$ be the genus of $\overline{C}$.

\begin{claim}
We can find $g+1$ distinct points $p_1, \ldots, p_{g+1}$ in $\bigcap_{i = 1}^n \phi_i(K) \subseteq C(K)$.
\end{claim}
\begin{claimproof}
For each $i$, choose a model $K_i$ of $T_i$ extending $K \restriction \mathcal{L}_i$.  Then $\phi_i(K_i)$ is Zariski dense in $C(K_i^{alg})$.  This (easily) implies that $\phi_i(K_i)^{g+1}$ is Zariski dense in $C^{g+1}(K_i^{alg})$.  If $U$ denotes the subset of $C^{g+1}$ consisting of $(x_1, \ldots, x_{g+1})$ such that $x_i \ne x_j$ for every $i$ and $j$, then $U$ is a Zariski dense Zariski open subset of $C^{g+1}$, because its complement is a closed subvariety of lower dimension.  The intersection of a Zariski dense set with a Zariski dense Zariski open is still Zariski dense.  So $\phi_i(K_i)^{g+1} \cap U$ is still Zariski dense in $C^{g+1}$.
Let $\psi_i(x_1, \ldots, x_{g+1})$ be the following quantifier-free $\mathcal{L}_i$-formula over $K$:
\[ \bigwedge_{j = 1}^{g+1} \phi_i(x_j) \wedge \bigwedge_{j \ne k} x_j \ne x_k.\]
Then $\psi_i(K_i) = \phi_i(K_i)^{g+1} \cap U$ is Zariski dense in $C^{g+1}(K_i^{alg})$, so $\psi_i(-)$ is $C^{g+1}$-dense.  By Axiom A2, it follows that some tuple $(p_1, \ldots, p_{g+1})$ satisfies
\[ \bigwedge_{i = 1}^n \psi_i(x_1, \ldots, x_{g+1}) \equiv \left(\bigwedge_{i = 1}^n \bigwedge_{j = 1}^{g+1} \phi_i(x_j) \right) \wedge \left(\bigwedge_{j \ne k} x_j \ne x_k\right).\]
Then $(p_1, \ldots, p_{g+1})$ has the desired properties.
\end{claimproof}

Now let $D$ be the divisor $\sum_j p_j$ on the curve $\overline{C}$.  By Riemann-Roch, $l(D) \ge \deg D + 1 - g = 2$.  The space of global sections of $\mathcal{O}(D)$ is a $K$-definable vector space of dimension at least 2.  Now $K$ is either algebraically closed or has characteristic zero, so $K$ is perfect.  Therefore, by Hilbert's Theorem 90 we know that this vector space has a $K$-definable basis (see Remark~\ref{hilbert90explanation}).  Thinking of the sections of $\mathcal{O}(D)$ as functions with poles no worse than $D$, we can find a non-constant meromorphic function $g$, with $(g) - D \ge 0$.  Then the divisor of poles of $g$ is a subset of $D$, so every pole of $g$ has multiplicity 1 and is in $\bigcap_{i = 1}^n \psi_i(K)$.  Take $f = 1/g$.
\end{proof}

\begin{proof}[Proof (of Theorem~\ref{miracle})]
If $K \models T$, then $K$ satisfies Axioms A1 and A2.  Axiom A1 implies that $K$ is algebraically closed or real closed or $p$-adically closed (Remark~\ref{a1-simplification}).  As $K$ is existentially closed, it is also reasonably clear that all the named valuations must be non-trivial.  Consequently $K \restriction \mathcal{L}_1 \models T_1$ and $v_2, \ldots, v_n$ are non-trivial.  Lastly, suppose $T_i$ and $T_j$ induce the same topology on $K$ for some $i$.  For notational simplicity assume $i = 1$ and $j = 2$.  As the topologies are Hausdorff, we can find non-empty $U_1$ and $U_2$ with $U_1$ a $T_1$-open, $U_2$ a $T_2$-open, and $U_1 \cap U_2 = \emptyset$.  Since the topologies from $T_1$ and $T_2$ have a basis of open sets consisting of quantifier-free definable sets, we can shrink $U_1$ and $U_2$ a little, and assume $U_1$ is quantifier-free definable in $\mathcal{L}_1$ and $U_2$ is quantifier-free definable in $\mathcal{L}_2$.  Now $U_1$ and $U_2$ are both Zariski dense in the affine line, so the formulas defining $U_1$ and $U_2$ are $\mathbb{A}^1$-dense.  Hence, by Axiom A2, $U_1$ must intersect $U_2$, a contradiction.

The other direction of the theorem is harder.  We proceed by induction on $n$, the number of orderings and valuations.  The base case where $n = 1$ is easy/trivial, so suppose $n > 1$.  Suppose $K$ satisfies the assumptions of the Theorem.  By Fact~\ref{strong-approximation}, we know that the $n$ different topologies on $K^1$ are independent.  The first bullet point ensures that $K$ satisfies axiom A1.  By Theorem~\ref{curve-reduction}, it suffices to prove axiom A2(1).  By Remark~\ref{combo-restatement}, we merely need to prove the following:
\begin{quote}
Let $C$ be a geometrically irreducible smooth affine curve defined over $K$.  Let $\phi_1(x)$ be a quantifier-free $\mathcal{L}_1$-formula with parameters from $K$ such that $\phi_1(K)$ is a non-empty open subset of $C$.  For $2 \le i \le n$, let $\phi_i(x)$ be a quantifier-free $\mathcal{L}_i$-formula with parameters from $K$ such that $\phi_i(x)$ defines a non-empty open subset of $C(K^{alg})$ with respect to any/every extension of the $i$th valuation $v_i$ from $K$ to $K^{alg}$.  THEN $\bigcap_{i = 1}^n \phi_i(K)$ is non-empty.
\end{quote}
Here we are using the facts that $K \restriction \mathcal{L}_1$ is already a model of $T_1$, and that for $i > 1$, the field $K^{alg}$ with any extension of $v_i$ will be a model of $T_i = \ACVF$.

For $1 < i \le n$, choose some extension $v'_i$ of the valuation $v_i$ to $K^{alg}$.
\begin{claim}\label{uno}
$K$ is dense in $K^{alg}$ with respect to the $v'_i$-adic topology on $K^{alg}$.
\end{claim}
\begin{claimproof}
The claim is trivial if all the $T_i$ are ACVF, in which case $K = K^{alg}$.  So we may assume characteristic zero.  It suffices to show that $K$ is dense in every finite Galois extension of $K$.\footnote{Note that the value group $v'_i(K)$ is cofinal in $v'_i(K^{alg})$, so e.g. the $v_i$-adic topology on $K$ is the restriction of the $v'_i$-adic topology on $K^{alg}$ to $K$.  Various pathologies are thus avoided.}  Let $L/K$ be a finite Galois extension.  We can write $L$ as $K(\zeta)$ for some singleton $\zeta$.  Let $P(X) \in K[X]$ be the minimal polynomial of $\zeta$ over $K$.  The function $x \mapsto P(x)$ from $K$ to $K$ is finite-to-one, so it has infinite image.  As $K$ is a model of ACVF, $p$CF, or RCF, we see by Fact~\ref{the-fact} that the image $P(K)$ of this map contains an open subset of $K$ with respect to the $T_1$-topology.  Because the $v_i$-adic topology on $K$ is independent from the $T_1$-topology on $K$, we can find elements of $P(K)$ of arbitrarily high $v$-valuation.  By the cofinality of the value groups, for every $\gamma \in v'_i(K^{alg})$, we can find an $x \in K$ with $v_i(P(x)) > \gamma$.  Let $\zeta_1, \ldots, \zeta_m \in L$ be the conjugates of $\zeta$ over $K$, counted with multiplicities.  Then we have just seen that for any $\gamma \in v'_i(K^{alg})$, we can find an $x \in K$ with
\[ \gamma < v'_i(P(x)) = \sum_{ i = m}^n v'_i(x - \zeta_i).\]
This implies that at least one of the $\zeta_i$'s is in the topological closure of $K$ with respect to $v'_i$.  Consequently, the $v'_i$-topological closure of $K$ in $L$ must contain $K[\zeta_i]$ for some $i$.  But $K[\zeta_i] = L$, so $K$ is $v'_i$-dense in $L$.
\end{claimproof}

Now suppose we are given a geometrically irreducible smooth affine curve $C$ defined over $K$, and we have $\mathcal{L}_i$-formulas $\phi_i(x)$ with parameters from $K$, such that $\phi_1(K)$ is a non-empty open subset of $C(K)$, and for $1 < i \le n$, $\phi_i(K^{alg})$ is a non-empty $v_i'$-open subset of $C(K^{alg})$.  (Here we are interpreting $\phi_i(K^{alg})$ using $v_i'$.)  By the inductive hypothesis, $K \restriction \bigcup_{i < n} \mathcal{L}_i$ is an existentially closed model of $\bigcup_{i < n} (T_i)_\forall$.  Applying Lemma~\ref{jacobian} to it, we can find a $K$-definable rational function $f : C \to \mathbb{P}^1$, whose divisor of zeros has no multiplicities and consists entirely of points in $\bigcap_{i < n} \phi_i(K)$ (and no points at infinity and no points in $C(K^{alg}) \setminus C(K)$).  Write this divisor as $\sum_{j = 1}^m (P_j)$, where the $P_j$ are $m$ distinct points in $\bigcap_{i < n} \phi_i(K)$.  Note that $m$ is the degree of $f$.

\begin{claim}\label{dos} There is a $T_1$-open neighborhood $U \subseteq K$ of zero such that for every $y \in U_1$, the divisor $f^{-1}(y)$ consists of $j$ distinct points in $\phi_1(K)$.  In particular, it contains no points in $C(K^{alg}) \setminus C(K)$ and no points in $\overline{C} \setminus C$.
\end{claim}
\begin{claimproof}
Because the $P_j$ are distinct, they have multiplicity one, so $f$ does not have a critical point at any of the $P_j$'s.  Consequently, by the implicit function theorem there is a $T_1$-open neighborhood $W_j \subseteq C(K)$ of $P_j$ such that $f$ induces a $T_1$-homeomorphism from $W_j$ to an open neighborhood of $0$.  By shrinking $W_j$ if necessary, we may assume that $W_j \subseteq \phi_1(K)$, and that $W_j \cap W_{j'} = \emptyset$ for $j \ne j'$.  Now let $U = \bigcap_{j = 1}^m f(W_j)$.  This is an open neighborhood of $0$ in the affine line $K^1$.  And if $y \in U$, then $f^{-1}(y)$ contains at least one point in each $W_j$.  Since the $W_j$ are distinct, these points are distinct.  Since $f$ is a degree-$m$ map, this exhausts the divisor $f^{-1}(y)$.
\end{claimproof}

\begin{claim}\label{trees} For $1 < i < n$, there is a $\gamma_i \in v_i(K)$ such that if $y \in K^{alg}$ and $v_i'(y) > \gamma_i$, then $f^{-1}(y)$ are all in $\phi_i(K^{alg})$.
\end{claim}
\begin{claimproof}
Use the same argument as Claim~\ref{dos}.
\end{claimproof}

By Claim~\ref{uno}, $K$ is dense in $K^{alg}$ with respect to the $v'_n$-adic topology.  Also, by assumption, $\phi_n(x)$ interpreted in $(K^{alg},v'_n)$ yields a non-empty $v'_n$-open subset $W \subseteq C(K^{alg})$.  Since $f$ is finite-to-one, the image $f(W)$ is an infinite subset of $\mathbb{P}^1(K^{alg})$, hence it has non-empty $v'_n$-interior.  Let $V$ be a $v'_n$-open subset of $\mathbb{P}^1(K^{alg})$ contained in $f(W)$.
Now, as $K$ is $v'_n$-adically dense in $K^{alg}$, $V$ must intersect $K$.  In particular, $V \cap K$ is a \emph{non-empty} $v_n$-adic open subset of $K$.  By independence of the topologies, we can find a $y$ in $\mathbb{A}^1(K)$ such that
\begin{itemize}
\item $y$ is in $U$, the $T_1$-open neighborhood of $0$ from Claim~\ref{dos}.
\item $v_i(y) > \gamma_i$, for $1 < i < n$, where the $\gamma_i$ are from Claim~\ref{trees}
\item $y$ is in $V \cap K$.
\end{itemize}
Having chosen such a $y$, we know by Claim~\ref{dos} that $f^{-1}(y)$ consists of $j$ distinct points in $\phi_1(K)$.  In particular, each point in $f^{-1}(y)$ is a point of $C(K)$.  And by Claim~\ref{trees}, each of these points also belongs to $\phi_i(K^{alg})$, hence satisfies $\phi_i(-)$, for $i < n$.
Finally, because $y$ is in $V \cap K$, $y$ is in the image of $\phi_n(K^{alg})$ under $f$.  So there is some $x \in \phi_n(K^{alg})$ mapping to $y$.  But we said that every point in $C(K^{alg})$ mapping to $y$ is already in $C(K)$ and even in $\bigcap_{i < n} \phi_i(K)$.  Thus
\[ x \in \phi_n(K^{alg}) \cap \bigcap_{i < n} \phi_i(K) = \bigcap_{i = 1}^n \phi_i(K).\]
In particular some point in $C(K)$ satisfies $\bigwedge_{i = 1}^n \phi_i(x)$, and the theorem is proven.
\end{proof}

\section{Keisler Measures}\label{probabilities}
To establish NTP$_2$ and analyze forking and dividing in $T$, we need the following tool.
\begin{theorem}\label{dream-thm}
Let $T$ be one of the model companions from \S\ref{review}.  For each $K \models T_\forall$ that is a perfect field, each formula $\phi(x)$ and each tuple $a$ from $K$, we can assign a number $P(\phi(a), K) \in [0,1]$ such that the following conditions hold:
\begin{itemize}
\item If $K$ is held fixed, the function $P(-, K)$ is a Keisler measure on the space of completions of the quantifier-free type of $K$.  Thus
\[ P(\phi(a), K) + P(\psi(b), K) = P(\phi(a) \wedge \psi(b), K) + P(\phi(a) \vee \psi(b), K)\]
\[ P(\neg \phi(a), K) = 1 - P(\phi(a), K)\]
for sentences $\phi(a)$ and $\psi(b)$ over $K$.  And if $\phi(a)$ holds in every model of $T$ extending $K$, then $P(\phi(a), K) = 1$.  For example, if $\phi(x)$ is quantifier-free, then $P(\phi(a), K)$ is 1 or 0 according to whether or not $K \models \phi(a)$.  And if $K$ is satisfies axiom A1 of \S\ref{axiom-section}, then $P(\phi(a), K) \in \{0,1\}$ for every $\phi(a)$, by Corollary~\ref{qe-version-1}.
\item Isomorphism invariance: if $K, L$ are two perfect fields satisfying $T_\forall$, and $f : K \to L$ is an isomorphism of structures, then $P(\phi(a), K) = P(\phi(f(a)), L)$ for every $K$-sentence $\phi(a)$.
\item Extension invariance: if $K_0 \subseteq K$ are perfect fields satisfying $T_\forall$, and $K_0$ is relatively algebraically closed in $K$, and $\phi(a)$ is a formula with parameters from $K_0$, then $P(\phi(a), K_0) = P(\phi(a), K)$.
\item Density: if $K \models T_\forall$ is a perfect field and $\phi(a)$ is a $K$-formula, and if $M \models \phi(a)$ for at least one $M \models T$ extending $K$, then $P(\phi(a), K) > 0$.  In other words, the associated Keisler measure is spread out throughout the entire Stone space of completions of $\qftp(K)$.
\end{itemize}
\end{theorem}

\subsection{The Algebraically Closed Case}

We first prove Theorem~\ref{dream-thm} in the case where every $T_i$ is a model of ACVF, i.e., the case of existentially closed fields with $n$ valuations.  Define $P(\phi(a), K)$ as follows.  Fix some algebraic closure $K^{alg}$ of $K$.  For each $1 \le i \le n$, let $v_i'$ be an extension to $K^{alg}$ of the $i^{\text{th}}$ valuation $v_i$ on $K$.  Choose automorphisms $\sigma_1, \ldots, \sigma_n \in \Gal(K^{alg}/K)$ randomly with respect to Haar measure on $\Gal(K^{alg}/K)$.  Then
\[ K_{\sigma_1, \ldots, \sigma_n} := (K^{alg}, v'_1 \circ \sigma_1, v'_2 \circ \sigma_2, \ldots, v'_n \circ \sigma_n)\]
is a model of $T_\forall$ satisfying axiom A1 of \S\ref{axiom-section}.  In particular, whether or not $\phi(a)$ holds in a model of $T$ extending $K_{\sigma_1,\ldots,\sigma_n}$ does not depend on the choice of the model, by Corollary~\ref{qe-version-1}.  Define $P(\phi(a), K)$ to be the probability that $\phi(a)$ holds in any/every model of $T$ extending $K_{\sigma_1, \ldots, \sigma_n}$.  This probability exists, i.e., the relevant event is measurable, because whether or not $\phi(a)$ holds is determined by the behavior of the valuations on some finite Galois extension $L/K$, by virtue of Theorem~\ref{qe-version-2}.

Note that the choice of the $v'_i$ does not matter.  If $v$ is a
valuation on $K$ and $w_1$ and $w_2$ are two extensions of $v$ to
$K^{alg}$, then there is a $\tau$ in $\Gal(K^{alg}/K)$ such that $w_1
= w_2 \circ \tau$.  
Thus, if $\sigma$ is a randomly chosen element of
$\Gal(K^{alg}/K)$, then $w_1 \circ \sigma$ and $w_2 \circ \sigma$ have
the same distribution.  Consequently the choice of the valuations
$v'_i$ does not effect the resulting value of $P(\phi(a), K)$.

So we have a well-defined number $P(\phi(a), K)$, and it is defined canonically.  The first two bullet points of Theorem~\ref{dream-thm} are therefore clear.  The density part can be seen as follows: suppose $M \models \phi(a)$ for some $M \models T$ extending $K$.  Let $K^{alg}$ be the algebraic closure of $K$ in $M$.  For the $v'_i$'s, take the restrictions of the valuations on $M$ to $K^{alg}$.  By Theorem~\ref{qe-version-2}, there is a field $K \le L \le K^{alg}$ with $L/K$ a finite Galois extension, such that $\phi(a)$ is implied by $T$ and the diagram of $L$.  Specifically, write $\phi(a)$ as $\exists y  : Q(y;a) = 0 \wedge \psi(y;a)$, and let $L$ be the splitting field of the polynomial $Q(X;a) \in K[X]$.  Now with probability $1/[L:K]^n$, every $\sigma_i$ will restrict to the identity on $L$.  Consequently, $K_{\sigma_1, \ldots, \sigma_n}$ will be a model of $T_\forall$ extending $L$, so in any model $M$ of $T$ extending $K_{\sigma_1, \ldots, \sigma_n}$, $\phi(a)$ will hold.  So $\phi(a)$ holds with probability at least $1/[L:K]^n$, and consequently $P(\phi(a),K) \ge 1/[L:K]^n$.

It remains to verify the extension invariance part of Theorem~\ref{dream-thm}.
Let $K_0 \le K$ be an inclusion of perfect fields, with $K_0$
relatively algebraically closed in $K$.  Let $\phi(a)$ be a formula
with parameters $a$ from $K_0$.  As in the previous paragraph, write
$\phi(a)$ as $\exists y : Q(y;a) = 0 \wedge \psi(y;a)$ and let $L_0$ be
the splitting field of $Q(y;a)$ over $K_0$.  At present $L_0$ is
nothing but a pure field.  Write $L_0 = K_0(\beta)$ for some singleton
$\beta \in L_0$, and let $Q(X)$ be the irreducible polynomial of
$\beta$ over $K_0$.  Let $L = L_0K = K(\beta)$; this is a Galois
extension of $K$.  There are only finitely many ways of factoring
$Q(X)$ in $K^{alg}$, so in each way of factoring $Q(X)$, the
coefficients come from $K_0^{alg}$.  In particular, if $Q(X)$ can be
factored over $K$, the coefficients would belong to $K_0^{alg} \cap K
= K_0$.  So $Q(X)$ is still irreducible over $K$.  Consequently $[L :
  K] = \deg Q(X) = [L_0 : K_0]$.  Now there is a natural restriction
map $\Gal(L/K) \to \Gal(L_0/K_0)$.  It is injective because an element
of $\Gal(K(\beta)/K)$ is determined by what it does to $\beta$.  Since
$\Gal(L/K)$ has the same size as $\Gal(L_0/K_0)$, the restriction map
must be an isomorphism.  Consequently, if $\tau$ is chosen from
$\Gal(L/K)$ randomly, its restriction to $L_0$ is a random element of
$\Gal(L_0/K_0)$.  Consequently, if $\sigma$ is a random element of
$\Gal(K^{alg}/K)$ and $\sigma_0$ is a random element of
$\Gal(K_0^{alg}/K_0)$, then $\sigma \restriction L_0$ and $\sigma_0
\restriction L_0$ have the same distribution, namely, the uniform
distribution on $\Gal(L_0/K_0)$.  From this, it follows easily that
$P(\phi(a), K) = P(\phi(a), K_0)$.

This completes the proof of Theorem~\ref{dream-thm} when every $T_i$ is ACVF.  The other cases are more complicated, though as a consolation all fields are characteristic zero, hence perfect.

A first attempt at defining $P(\phi(a), K)$ is as follows: fix some algebraic closure $K^{alg}$ of $K$.  For each $i$ such that $T_i$ is RCF, let $K_i$ be a real closure of $(K, <_i)$ inside $K^{alg}$.  For each $i$ such that $T_i$ is $p$CF, let $K_i$ be a $p$-adic closure of $(K, v_i)$ inside $K^{alg}$.  For each $i$ such that $T_i$ is ACVF, let $K_i$ be $K^{alg}$ with some valuation extending $v_i$.  In each case, there is a choice, but any two choices are related by an element of $\Gal(K^{alg}/K)$.  Now choose $\sigma_1, \ldots, \sigma_n \in \Gal(K^{alg}/K)$ randomly.  For each $i$, consider $\sigma_i(K_i)$, which is (usually) a model of $T_i$ extending $K$.  Let $K'$ be the field
\[ K' = \bigcap_{i = 1}^n \sigma_i(K_i).\]
There is an obvious way to give $K'$ the structure of a $T_\forall$-model.  If we knew that $K'$ satisfies condition A1 of \S\ref{axiom-section} with high probability, we could define $P(\phi(a), K)$ to be the probability that $\phi(a)$ holds in any/every model of $T$ extending $K'$.  Unfortunately, $K'$ usually satisfies condition A1 with probability \emph{zero}.  Instead, we will proceed by repeating the above procedure with $K'$ in place of $K$, getting a third field $K''$.  Iterating this, we get an increasing sequence $K \subseteq K' \subseteq K'' \subseteq \cdots \subseteq K^{(n)} \subseteq \cdots$ of $T_\forall$-structures on subfields of $K^{alg}$.  The union $K^\infty = \bigcup_{n = 1}^\infty K^{(n)}$ does actually turn out to satisfy axiom A1 with probability 1, and we let $P(\phi(a), K)$ be the probability that $\phi(a)$ holds in any/every model of $T$ extending $K^\infty$.

The rest of this section will make this construction more precise, and verify that it satisfies the requirements of Theorem~\ref{dream-thm}.

\subsection{The General Case}
All fields will be perfect, unless stated otherwise.  All models of $T_\forall$ and $(T_i)_\forall$ will be (perfect) fields, unless stated otherwise.  Galois extensions need not be finite Galois extensions.

We start off with some easy but confusing facts that will be needed later.
\begin{lemma}\label{loc-closed}
Let $L/K$ be a Galois extension of fields, and suppose $K$ has the structure of a $(T_i)_\forall$ model (but $L$ does not).  The following are equivalent
\begin{description}
\item[(a)] For every $F$, if $F$ is a model of $(T_i)_\forall$ extending $K$, and $F$ is a subfield of $L$, then $F = K$.
\item[(b)] There is a model $M \models T_i$ extending $K$, such that $M \cap L = K$.
\item[(c)] For every model $M \models T_i$ extending $K$, $M \cap L = K$.
\end{description}
Note that it makes sense to talk about whether $M \cap L = K$, because $L/K$ is Galois.
\end{lemma}
\begin{proof}
The equivalence of (b) and (c) follows from quantifier elimination in $T_i$.  Indeed, the statement that $M \cap L = K$ is equivalent to the statement that for each $x \in L \setminus K$, the irreducible polynomial of $x$ over $K$ has no zeros in $M$.  This is a conjunction of first order statements about $K$, so it holds in one choice of $M$ if and only if it holds in another choice of $M$.

Suppose (a) holds.  Let $M$ be a model of $T_i$ extending $K$.  Taking $F = M \cap L$, (a) implies that $M \cap L = K$.  So (a) implies (c).

Conversely, suppose (a) does not hold.  Let $F$ witness a contradiction to (a), so $K \subsetneq F \subseteq L$, and $F$ is a model of $(T_i)_\forall$ extending $K$.  Let $M$ be a model of $T_i$ extending $F$ and hence $K$.  Then $M \cap L$ contains $F$, contradicting (c).
\end{proof}

\begin{definition}
Say that $K$ is \emph{locally $T_i$-closed in $L$} if it satisfies the equivalent conditions of the previous lemma.
\end{definition}
\begin{definition}
Let $L/K$ be a Galois extension of fields, and suppose $K$ has the structure of a $(T_i)_\forall$-model (but $L$ does not).  Let $\mathfrak{C}_i(L/K)$ denote the set of models of $(T_i)_\forall$ which extend $K$, are subfields of $L$, and are locally $T_i$-closed in $L$.
\end{definition}
The subscript on $\mathfrak{C}_i$ is present so that
$\mathfrak{C}_i(L/K)$ will be unambiguous when $K$ is a model of
$T_\forall$, in addition to being a model of $(T_i)_\forall$.

There is a natural action of $\Gal(L/K)$ on $\mathfrak{C}_i(L/K)$.

\begin{lemma}\label{random-confusion}
Suppose $L/K$ is a Galois extension of fields, and $K \models (T_i)_\forall$.
\begin{description}
\item[(a)] The action of $\Gal(L/K)$ on $\mathfrak{C}_i(L/K)$ has exactly one orbit.
\item[(b)] Suppose $K'$ is a model of $(T_i)_\forall$ extending $K$, and $L'$ is a field extension of $L$ and $K'$, with $L'$ Galois over $K'$.  If $F \in \mathfrak{C}_i(L'/K')$, then $F \cap L \in \mathfrak{C}_i(L/L \cap K')$. 
\end{description}
\end{lemma}
\begin{proof}
\begin{description}
\item[(a)]
Note that $\mathfrak{C}_i(L/K)$ is non-empty by a Zorn's lemma argument and condition (a) of Lemma~\ref{loc-closed}.
Now suppose $F$ and $F'$ are two elements of $\mathfrak{C}_i(L/K)$.  By quantifier elimination in $T_i$, we can amalgamate $F$ and $F'$ over $K$.  Thus, we can find a model $M \models T$ extending $F$, and an embedding $\iota : F' \to M$ which is the identity on $K$.  Choosing some way of amalgamating $M$ and $L$ as fields, we get that $\iota(F') \subseteq L \supseteq F$, because of $L/K$ being Galois.  The compositum $\iota(F')F$ is a subfield of $L$ with a $(T_i)_\forall$-structure extending that on $F$ and $\iota(F')$, so by local $T_i$-closedness of $\iota(F')$ and $F$ in $L$, $\iota(F') = \iota(F')F = F$.  It follows that $F'$ and $F$ are isomorphic over $K$.  This isomorphism must extend to an automorphism of $L$, because $L/K$ is Galois.  So some automorphism on $L/K$ maps $F'$ to $F$ (as $(T_i)_\forall$-structures).
\item[(b)] Let $M$ be a model of $T_i$ extending $F$.  Choose some way of amalgamating $M$ with $L'$.  Then $M \cap L' = F$ by (c) of Lemma~\ref{loc-closed}.  Therefore, $M \cap L = M \cap L' \cap L = F \cap L$.  So by (b) of Lemma~\ref{loc-closed}, $F \cap L$ is locally $T_i$-closed in $L$.  Therefore it is in $\mathfrak{C}_i(L/L \cap K')$.  \qedhere
\end{description}
\end{proof}

Now we turn our attention from $T_i$ to $T$.
\begin{definition}
Let $K \models T_\forall$ and let $L$ be a pure field that is a Galois extension of $K$.  Let $\mathfrak{S}(L/K)$ be the set of all $K' \models T_\forall$ extending $K$, with $K'$ a subfield of $L$.  In other words, an element of $\mathfrak{S}(L/K)$ is a subfield $F$ of $L$, endowed with a $T_\forall$-structure, such that $F \supseteq K$ and the structure on $F$ extends the structure on $K$.
\end{definition}
There is a natural partial order on $\mathfrak{S}(L/K)$ coming from inclusion of substructures.  There is also a natural action of $\Gal(L/K)$ on $\mathfrak{S}(L/K)$.  One should think of $\mathfrak{S}(L/K)$ as the set of states in a Markov chain, specifically the random process described at the end of the previous section.

\begin{definition}
Suppose $K \models T_\forall$ and $L/K$ is a Galois extension of $K$.  For $1 \le i \le n$, choose some $L_i \in \mathfrak{C}_i(L/K)$.
Choose $\sigma_1, \ldots, \sigma_n \in \Gal(L/K)$ independently and randomly, using Haar measure on $\Gal(L/K)$.  Let $F$ be $\bigcap_{i = 1}^n \sigma_i(L_i)$, with the obvious choice of a $T_\forall$ structure.  So $F$ is a random variable with values in $\mathfrak{S}(L/K)$.  Let $\mu^1_{L/K}$ be the probability distribution on $\mathfrak{S}(L/K)$ obtained in this way.  The choice of the $L_i$'s is irrelevant, by Lemma~\ref{random-confusion}(a).
\end{definition}
The superscript $1$ is to indicate that this is the first step of the Markov chain.

\begin{lemma}\label{probability-epsilon}
Suppose $L/K$ is finite.  Then every event (subset of $\mathfrak{S}(L/K)$) which has positive probability with respect to $\mu^1_{L/K}$ has probability at least $1/m^n$, where $m = [ L : K ]$.
\end{lemma}
\begin{proof}
The only randomness comes from the $\sigma_i$'s.  Each element of $\Gal(L/K)$ has an equal probability under Haar measure, and this probability is $1/m$.  Since the $\sigma_i$'s are chosen independently, each choice of the $\sigma_i$'s has probability $1/m^n$ of occurring.
\end{proof}

\begin{lemma}\label{move-on}
Suppose $L/K$ is finite, and $F$ is a maximal element of $\mathfrak{S}(L/K)$.  Then $\mu^1_{L/K}(\{F\}) > 0$.
\end{lemma}
\begin{proof}
For each $i$, let $M_i$ be a model of $T_i$ extending $F \restriction \mathcal{L}_i$, and choose some way of amalgamating $M_i$ and $L$ as fields over $F$.  Let $F_i = L \cap M_i$.  Of course $F_i \supseteq F$.  By Lemma~\ref{loc-closed}(b), $F_i \in \mathfrak{C}_i(L/K)$.  Let $F' = \bigcap_{i = 1}^n F_i$.  Then $F' \in \mathfrak{S}(L/K)$ and $F'$ extends $F$, so $F = F'$ by maximality of $F$.  Now if we choose $\sigma_1, \ldots, \sigma_n \in \Gal(L/K)$ randomly, then $\bigcap_{i =1 }^n \sigma_i(F_i)$ is distributed according to $\mu^1_{L/K}$.  Since $L/K$ is finite, there is a positive probability that $\sigma_i = 1$ for every $i$, in which case $\bigcap_{i = 1}^n \sigma_i(F_i) = F' = F$.
\end{proof}

\begin{lemma}\label{local-calc-0}
Let $L/K$ be a Galois extension, and $K$ be a model of $T_\forall$.  Let $K'$ be a model of $T_\forall$ extending $K$.  Let $L'$ be a field extending $L$ and $K'$, with $L'$ Galois over $K'$.  If $F \in \mathfrak{S}(L'/K')$ is distributed randomly according to $\mu^1_{L'/K'}$, then $F \cap L$ is distributed randomly according to $\mu^1_{L/L \cap K'}$.
\end{lemma}
\begin{proof}
For $1 \le i \le n$, choose some $F_i \in \mathfrak{C}_i(L'/K')$.  By Lemma~\ref{random-confusion}(b), $F_i \cap L$ is in $\mathfrak{C}_i(L/L \cap K')$.
\begin{claim}
If we choose $\sigma$ from $\Gal(L'/K')$ randomly using Haar measure, then $\sigma \restriction L$ is also randomly distributed in $\Gal(L/L \cap K')$ with respect to Haar measure.
\end{claim}
\begin{claimproof}
  If $\Pi$ denotes the image of the restriction homomorphism
  $\Gal(L'/K') \to \Gal(L/L \cap K')$, then the fixed field of $\Pi$
  is clearly $L \cap \dcl(K') = L \cap K'$ (all fields are perfect).
  By Galois theory, $\Pi = \Gal(L/L \cap K')$, and the restriction
  homomorphism is surjective.
\end{claimproof}  
From the Claim, we conclude that if the $\sigma_i$ are distributed randomly from $\Gal(L'/K')$, then $\sigma_i \restriction L$ are distributed randomly from $\Gal(L/L \cap K')$.  Taking $F = \bigcap_{i = 1}^n \sigma_i(F_i)$, we get $F$ distributed according to $\mu^1_{L'/K'}$.  But
\[ F \cap L = \bigcap_{i = 1}^n (\sigma_i \restriction L)(F_i \cap L)\]
is then distributed according to $\mu^1_{L/L \cap K'}$, because $F_i \cap L \in \mathfrak{C}_i(L/(K' \cap L))$ and $\sigma_i \restriction L$ is distributed according to Haar measure on $\Gal(L/L \cap K')$.
\end{proof}

\begin{definition}\label{markov-chain}
Let $L/K$ be a Galois extension, and $K$ be a model of $T_\forall$.  Define a series of distributions $\{\mu^i_{L/K}\}_{i < \omega}$ on $\mathfrak{S}(L/K)$ as follows:
\begin{itemize}
\item $\mu^0_{L/K}$ assigns probability 1 to $\{K\} \subseteq \mathfrak{S}(L/K)$.
\item $\mu^1_{L/K}$ is as above.
\item For $i > 0$, if we choose $F \in \mathfrak{S}(L/K)$ randomly according to $\mu^i_{L/K}$, and then choose $F' \in \mathfrak{S}(L/F) \subseteq \mathfrak{S}(L/K)$ randomly according to $\mu^1_{L/F}$, then $F'$ is distributed according to $\mu^{i+1}_{L/K}$.
\end{itemize}
\end{definition}
In other words, we are running some kind of Markov chain whose states are the elements of $\mathfrak{S}(L/K)$.  The transition probabilities out of the state $F$ are given by $\mu^1_{L/F}$, and $\mu^n_{L/K}$ is the distribution of the Markov chain after $n$ steps.


\begin{lemma}
Let $L/K$ be a finite Galois extension, and $K$ be a model of $T_\forall$.  Then $\lim_{i \to \infty} \mu^i_{L/K}$ exists, and the corresponding distribution on $\mathfrak{S}(L/K)$ is concentrated on the maximal elements of $\mathfrak{S}(L/K)$.
\end{lemma}
\begin{proof}
The fact that the limit distribution exists is a general fact about Markov chains with finitely many states such that the graph of possible transitions has no cycles other than self-loops.

It remains to check that in the limit, we land in a maximal element of $\mathfrak{S}(L/K)$ with probability one.
Let $m = [L : K]$.  If $F \in \mathfrak{S}(L/K)$ is not maximal, then the probability of moving from $F$ to some bigger element is positive by Lemma~\ref{move-on}, and at least $1/m^n$, by Lemma~\ref{probability-epsilon}.  The probability of getting stuck at $F$ is therefore bounded above by $\lim_{k \to \infty} (1 - 1/m^n)^k = 0$.  As there are finitely many non-maximal $F$, we conclude that the probability of getting stuck at any of them is zero.
\end{proof}
We let $\mu^\infty_{L/K}$ denote the limit distribution on $\mathfrak{S}(L/K)$.

\begin{lemma}\label{terminal-positive}
Let $L/K$ be a finite Galois extension, and $K$ be a model of $T_\forall$.  Then every maximal element of $\mathfrak{S}(L/K)$ has a positive probability with respect to $\mu^\infty_{L/K}$.
\end{lemma}
\begin{proof}
This follows immediately from Lemma~\ref{move-on}, and the fact that once the Markov chain reaches a maximal element of $\mathfrak{S}(L/K)$, it must remain there.
\end{proof}

\begin{lemma}\label{local-calc}
Let $L/K$ be a Galois extension, with $K$ a model of $T_\forall$.  Let $K'$ be a model of $T_\forall$ extending $K$.  Let $L'$ be a field extending $K'$ and $L$, Galois over $K'$.  If $F$ is a random element of $\mathfrak{S}(L'/K')$ distributed  according to $\mu^i_{LK'/K'}$, then $F \cap L$ is distributed according to $\mu^i_{L/L \cap K'}$.
\end{lemma}
\begin{proof}
We proceed by induction on $i$.  For $i = 0$, $F$ is guaranteed to be $K'$, and $F \cap L$ is guaranteed to be $K' \cap L$, which agrees with $\mu^0_{L/L \cap K'}$.

For the inductive step, suppose we know the statement of the lemma for $\mu^i$, and prove it for $\mu^{i+1}$.  If we let $F \in \mathfrak{S}(L'/K')$ be chosen according to $\mu^i_{L'/K'}$, and we then choose $F' \in \mathfrak{S}(L'/F) \subseteq \mathfrak{S}(L'/K')$ according to $\mu^1_{L'/F}$, then $F'$ is randomly distributed according to $\mu^{i+1}_{L'/K'}$, by definition of $\mu^{i+1}$.  Also, $F \cap L$ is distributed according to $\mu^i_{L/L \cap K'}$, by the inductive hypothesis.  By Lemma~\ref{local-calc-0} we know that $F' \cap L$ is distributed according to $\mu^1_{L/L \cap F}$.  In particular, the distribution of $F' \cap L$ only depends on $F \cap L$.  So if we want to sample $F' \cap L$, we can simply choose $F \cap L$ using $\mu^i_{L/L \cap K'}$, and can then choose $F' \cap L$ using $\mu^1_{L/F \cap L}$.  This is the recipe for sampling the distribution $\mu^{i+1}_{L/K' \cap L}$.  So $F' \cap L$ is indeed distributed according to $\mu^{i+1}_{L/K' \cap L}$.
\end{proof}

\begin{corollary}\label{loc-calc-infty}
When $L/K$ and $L'/K'$ are finite Galois extension, Lemma~\ref{local-calc} holds for $i = \infty$.
\end{corollary}

\begin{definition}
Let $K \models T_\forall$ be a perfect field, $\phi(a)$ be a formula in the language of $T$ with parameters $a$ from $K$.  Say that a finite Galois extension $L/K$ \emph{determines the truth of $\phi(a)$} if the following holds: whenever $M$ and $M'$ are two models of $T$ extending $K$, if $M \cap L$ is isomorphic as a model of $T_\forall$ to $M' \cap L$, then $[M \models \phi(a)] \iff [M' \models \phi(a)]$.  (Note that the isomorphism class of $M \cap L$ does not depend on how we choose to form the compositum $ML$.)
\end{definition}
For every formula $\phi(a)$, there \emph{is} some finite Galois extension $L/K$ which determines the truth of $\phi(a)$.  Namely, use Theorem~\ref{qe-version-2} to write $\phi(a)$ in the form $\exists y :  Q(y;a) = 0 \wedge \psi(y;a)$, and take $L$ to be the splitting field over $K$ of $Q(X;a) \in K[X]$.

\begin{lemma}\label{yet-another-lemma}
Let $K$ be a model of $T_\forall$, $M$ be a model of $T$ extending $K$, and let $L/K$ be a Galois extension of $K$.  Assume $M$ and $L$ are embedded over $K$ into some bigger field.  Then $M \cap L$ is a maximal element of $\mathfrak{S}(L/K)$.
\end{lemma}
\begin{proof}
Suppose not.  Let $F$ be an element of $\mathfrak{S}(L/K)$, strictly bigger than $M \cap L$, and finitely generated over $M \cap L$.  Let $x$ be a generator of $F$ over $M \cap L$.  If $S$ denotes the set of algebraic conjugates of $x$ over $M$, then the code for the finite set $S$ is in $M$, and also in $L$ because $S \subseteq L$.  So the code for $S$ is in $M \cap L$, implying that $S$ is also the set of algebraic conjugates of $x$ over $M \cap L$.  Since we are assuming that all fields are perfect, this implies that the degree of $x$ over $M$ is the same as the degree of $x$ over $M \cap L$.  In particular, the irreducible polynomial $Q(X)$ of $x$ over $M \cap L$ remains irreducible over $M$.
For $1 \le i \le n$, let $M_i$ be a model of $T_i$ extending $M \restriction \mathcal{L}_i$.  Let $N_i$ be a model of $T_i$ extending $F \restriction \mathcal{L}_i$.  The polynomial $Q(X)$ has a zero in $F$, namely $x$.  Hence it has a zero in $N_i \supseteq F$.  As $M_i$ and $N_i$ are two models of $T_i$ extending $M \cap L$ and $Q(X)$ is defined over $M \cap L$, it follows from quantifier-elimination in $T_i$ that $Q(X)$ also has a zero in $M_i$.

Now we have a polynomial $Q(X)$ of degree $> 1$, irreducible over $M$, such that $Q(X)$ has a root in $M_i$ for every $i$.  This contradicts condition A1' of \S \ref{axiom-section}.
\end{proof}

\begin{definition}
Let $L/K$ be a Galois extension, with $K$ a model of $T_\forall$.  Let $\mathfrak{F}(L/K)$ be the set of maximal elements of $\mathfrak{S}(L/K)$.
\end{definition}
By Zorn's lemma, it is clear that every element of $\mathfrak{S}(L/K)$ is bounded above by an element of $\mathfrak{F}(L/K)$, even if $L/K$ is infinite.  When $L/K$ is a finite extension, $\mu^\infty_{L/K}$ induces a probability distribution on $\mathfrak{F}(L/K)$.

\begin{remark}\label{nearly-there} $\mathfrak{F}(L/K)$ is exactly the set of $F$ of the form $L \cap M$, where $M$ is a model of $T$ extending $K$.  One inclusion is Lemma~\ref{yet-another-lemma}.  The other inclusion is obvious: if $F$ is a maximal element of $\mathfrak{S}(L/K)$, then letting $M$ be a model of $T$ extending $F$, and combining $M$ and $L$ into a bigger field in any way we like, we have $F \subseteq M \cap L \in \mathfrak{S}(L/K)$, so maximality of $F$ forces $M \cap L = F$.
\end{remark}

Suppose that $L/K$ determines the truth of $\phi(a)$.  Then by Remark~\ref{nearly-there}, there must be a uniquely determined map $f_{\phi(a),L}$ from $\mathfrak{F}(L/K)$ to $\{\bot,\top\}$ such that for every $M \models T$ extending $K$, and every way of forming the compositum $ML$, the truth of $M \models \phi(a)$ is given by $f_{\phi(a),L}(M \cap L)$.

Another corollary of Remark~\ref{nearly-there} is that if $K \le L \le L'$, with $L'$ and $L$ Galois extensions of $K \models T_\forall$, and if $F \in \mathfrak{F}(L'/K)$, then $F = M \cap L'$ for some model $M$, and hence $F \cap L = M \cap L' \cap L = M \cap L$ is in $\mathfrak{F}(L/K)$.

\textbf{Finally, we define $P(\phi(a), K)$ to be $\mu^\infty_{L/K}(\{F  : f_{\phi(a),L}(F) = \top\})$.}
\begin{lemma}\label{thats-good}
The choice of $L$ does not matter.
\end{lemma}
\begin{proof}
If $L$ and $L'$ are two finite Galois extensions of $K$ which determine the truth of $\phi(a)$, then so does their compositum $LL'$.  So we may assume $L \subseteq L'$.  Let $r : \mathfrak{F}(L'/K) \to \mathfrak{F}(L/K)$ be the restriction map, $F \mapsto F \cap L$.  
\begin{claim} $f_{\phi(a),L'} = f_{\phi(a),L} \circ r$.
\end{claim}
\begin{claimproof}
For $F \in \mathfrak{F}(L'/K)$, we will show $f_{\phi(a),L'}(F) = f_{\phi(a),L}(r(F))$.  Write $F$ as $M \cap L'$, with $M$ a model of $T$ extending $F$.  Then $f_{\phi(a),L'}(M \cap L') = f_{\phi(a),L'}(F)$ is the truth value of $M \models \phi(a)$.  But $M \cap L = F \cap L$, so by definition of $f_{\phi(a),L}$, we also know that $f_{\phi(a),L}(M \cap L) = f_{\phi(a),L}(r(F))$ is the truth value of $M \models \phi(a)$, which is the same thing.  So $f_{\phi(a),L}(r(F)) = f_{\phi(a),L'}(F)$.
\end{claimproof}

By Corollary~\ref{loc-calc-infty} applied in the $K = K'$ case, if $F$ is a random element of $\mathfrak{F}(L'/K)$ chosen according to $\mu^\infty_{L'/K}$, then $r(F) = F \cap L$ is distributed according to $\mu^\infty_{L/K}$.  In particular, the probability of $f_{\phi(a),L'}(F)$ or equivalently of $f_{\phi(a),L}(r(F))$ is the same as the probability of $f_{\phi(a),L}(F')$, with $F'$ chosen directly from $\mu^\infty_{L/K}$.  But the former probability is $P(\phi(a),K)$ computed using $L'$, while the latter is $P(\phi(a),K)$ computed using $L$.
\end{proof}

So $P(\phi(a), K)$ is at least a well-defined number.  The ``isomorphism invariance'' part of Theorem~\ref{dream-thm} is clear from the definitions.
We need to prove the other conditions of Theorem~\ref{dream-thm}.

\begin{lemma}
For any fixed $K$, the function $P( -, K)$ is a Keisler measure on the space of completions of the quantifier-free type of $K$.
\end{lemma}
\begin{proof}
It suffices to prove the following:
\begin{itemize}
\item If $\phi(a)$ and $\psi(b)$ are forced to be logically equivalent by $T$ and the diagram of $K$, then $P(\phi(a), K) = P(\psi(b), K)$.  This is easy/trivial, because if we choose a finite Galois extension $L$ determining the truth of both $\phi(a)$ and $\psi(b)$, we see that $f_{\phi(a),L} = f_{\psi(b),L}$ by unwinding the definitions.
\item $P(\phi(a), K) = 1 - P(\neg \phi(a), K)$, which follows similarly, though it uses the fact that $\mu^\infty_{L/K}$ is concentrated on $\mathfrak{F}(L/K)$.
\item If $\phi(a) \wedge \psi(b)$ contradicts $T \cup \diag(K)$, then $P(\phi(a) \vee \psi(b), K) = P(\phi(a), K) + P(\psi(b), K)$.  Again, this is not difficult: if $L$ is a field determining the truth of both $\phi(a)$ and $\psi(b)$, then it is clear that
\[ f_{\phi(a), L} \wedge f_{\psi(b), L} = f_{\phi(a) \wedge \psi(b), L} = \bot\]
\[ f_{\phi(a), L} \vee f_{\psi(b), L} = f_{\phi(a) \vee \psi(b), L}.\]
Consequently, $\{F : f_{\phi(a) \vee \psi(b), L}(F) = \top\}$ is a disjoint union of $\{F : f_{\phi(a), L}(F) = \top\}$ and $\{F : f_{\psi(b), L}(F) = \top\}$, so we reduce to the fact that $\mu^\infty$ is a probability distribution.
\item $0 \le P(\phi(a), K) \le 1$, which is clear from the definition. \qedhere
\end{itemize}
\end{proof}

\begin{lemma}
If $K \subseteq K'$ are models of $T_\forall$ and $K$ is relatively algebraically closed in $K'$, and $\phi(a)$ is a formula with parameters from $K$, then $P(\phi(a), K) = P(\phi(a), K')$.
\end{lemma}
(This is the ``extension invariance'' part of Theorem~\ref{dream-thm}.)
\begin{proof}
Let $L$ be a finite Galois extension of $K$ determining the truth of $\phi(a)$.  Let $L'$ be a finite Galois extension of $K'$ determining the truth of $\phi(a)$; we may assume $L' \supseteq L$.  (In fact, we can take $L' = LK'$.)  Because $K$ is relatively algebraically closed in $K'$, $L \cap K' = K$.  So by Corollary~\ref{loc-calc-infty}, if $F \in \mathfrak{F}(L'/K')$ is distributed according to $\mu^\infty_{L'/K'}$, then $F \cap L$ is distributed according to $\mu^\infty_{L/K}$.  Using Lemma~\ref{terminal-positive}, this implies that $F \cap L \in \mathfrak{F}(L/K)$ for any $F \in \mathfrak{F}(L'/K')$.  Let $r : \mathfrak{F}(L'/K') \to \mathfrak{F}(L/K)$ be the map $F \mapsto F \cap L$.  By unwinding the definitions (as in the claim in the proof of Lemma~\ref{thats-good}), one sees that $f_{\phi(a),L'/K'} = f_{\phi(a),L/K} \circ r$.  As in the proof of Lemma~\ref{thats-good}, we see that for $F'$ choosen randomly from $\mathfrak{F}(L'/K')$ and $F$ chosen randomly from $\mathfrak{F}(L/K)$, the distribution of $F$ and $r(F')$ is the same, and therefore so too is the distribution of
\[ f_{\phi(a),L'/K'}(F') = f_{\phi(a),L/K}(r(F'))\]
and
\[ f_{\phi(a),L/K}(F).\]
This ensures that $P(\phi(a), K) = P(\phi(a), K')$.
\end{proof}

\begin{lemma}
If $K \models T_\forall$ and $\phi(a)$ is a $K$-formula which holds in some model of $T$ extending $K$, then $P(\phi(a),K) > 0$.
\end{lemma}
(This is the ``density'' part of Theorem~\ref{dream-thm}.)
\begin{proof}
Let $M$ be the model where $\phi(a)$ holds.  Let $L$ be a Galois extension of $K$ determining the truth of $\phi(a)$.  Then $L \cap M \in \mathfrak{F}(L/K)$ and $f_{\phi(a),L}(L \cap M) = \top$.  By Lemma~\ref{terminal-positive}, $P(\phi(a),K) > 0$.  
\end{proof}

We have verified each condition of Theorem~\ref{dream-thm}, which is now proven.

\section{NTP$_2$ and the Independence Property}\label{classification}
We show that the model companion $T$ (usually) fails to be NIP, but is always NTP$_2$, the next best possibility.  
\subsection{Failure of NIP}\label{sec:nip}
If $n = 1$, then $T = T_i$ is one of ACVF, RCF, or $p$CF, which are all known to be NIP.  On the other hand,
\begin{theorem}\label{thm-nip}
Suppose $n > 1$.  Then $T$ has the independence property.
\end{theorem}
\begin{proof}
We give a proof which works in characteristic $\ne 2$.  It is not hard
to modify it to work in characteristic 2, replacing square roots with
Artin-Schreier roots.
\begin{claim} For each $i$, we can produce quantifier-free $\mathcal{L}_i$-formulas $\phi_i(x,y)$ and $\chi_i(y)$ without parameters such that $x, y$ are singletons, and such that if $K_i \models T_i$, then $\chi_i(K_i)$ is a non-empty open set and for every $b \in \chi_i(K_i)$, both square roots of $b$ are present in $K_i$, and exactly
one of them satisfies $\phi_i(x,b)$.
\end{claim}
\begin{claimproof}
If $T_i$ is RCF, let $\chi_i(y)$ say that $y > 0$ and $\phi_i(x,y)$ say that $x > 0$.
If $T_i$ is ACVF, let $\chi_i(y)$ say that $v(y - 1/4) > 0$, and $\phi_i(x,y)$ say that $v(x - 1/2) > 0$.  Note that if $v(y - 1/4) > 0$ and $x^2 = y$, then
$t = x - 1/2$ satisfies
\[ t^2 + t + 1/4 - y = (t + 1/2)^2 - y = 0.\]
By Newton polygons, one of the possibilities for $t$ has valuation zero, and the other has valuation $v(y - 1/4) > 0$.
If $T_i$ is $p$CF, the same formulas work as in the case of ACVF.  The only thing to check is that if $v(y - 1/4) > 0$ for some $y \in K \models \pCF$, then the two roots of $T^2 + T + (1/4 - y) = 0$ are present in $K$.  If not, then since the two roots have different valuations (in an ambient model of ACVF), there are two different ways to extend the valuation from $K$ to $K[T]/(T^2 + T + (1/4 - y))$, contradicting Henselianity of $K$.
\end{claimproof}
Given the $\phi_i$ and $\chi_i$ from the Claim, let $\chi(y) = \bigwedge_{i = 1}^n \chi_i(y)$.  Note that $\chi(y)$ defines an infinite subset of any model of $T$, by condition A2 of \S\ref{axiom-section}.  (Each $\chi_i(-)$ is $\mathbb{A}^1$-dense.)  If $K \models T$ and $b \in \chi(K)$, then $X^2 - b$ has roots in $K$ by choice of $\chi_i(-)$ and condition A1' of \S\ref{axiom-section}.  So each element of $\chi(K)$ is a square.

Let $\psi(y)$ assert that $\chi(y)$ holds and there is a square root
of $y$ which satisfies exactly one of $\phi_1$ and $\phi_2$.  Note
that if $\chi(b)$ holds, then both square roots of $b$ are present in
$K$, exactly one of them satisfies $\phi_1$, and exactly one of them
satisfies $\phi_2$.  Letting, $\oplus$ denote exclusive-or, we can
write $\psi(y)$ as $\phi_1(\sqrt{y}) \oplus \phi_2(\sqrt{y})$, where
the choice of $\sqrt{y}$ is unimportant.

Let $K$ be a model of $T$.  We will show that $\psi(x+y)$ has the independence property in $K$.  Let $a_1, \ldots, a_m$ be any $m$ elements in $\chi(K)$, which as we noted above is an infinite set.  We will show that for any subset $S_0 \subseteq \{1,\ldots,m\}$, there is a $b$ in $K$ such that $j \in S_0 \iff K \models \psi(b+a_j)$.  It suffices to find such a $b$ in an elementary extension of $K$, rather than $K$ itself.  Let $K' \succeq K$ be an elementary extension containing an element $\epsilon$ which is infinitesimal compared to $K$, with respect to every one of the valuations.  That is, for each $i$ such that $T_i$ is valuative, we want $v_i(\epsilon) > v_i(K)$, and for each $i$ such that $T_i$ is RCF, we want $-\alpha <_i \epsilon <_i \alpha$ for every $\alpha >_i 0$ in $K$.  The fact that such an $\epsilon$ exists follows by our axiom A2, and can be shown directly.

Note that for $1 \le j \le m$, $a_j + \epsilon \in \chi(K')$.  (Indeed, for every $i$, $K' \models \chi_i(a_i + \epsilon)$, because $\chi_i(-)$ defines an open set in a model of $T_i$, and $\epsilon$ is infinitesimal with respect to the prime model of $T_i$ over $K \restriction \mathcal{L}_i$.)  Consequently, $\sqrt{a_j + \epsilon} \in K'$ for every $1 \le j \le m$.
Let $L$ be $K(\sqrt{a_j + \epsilon} : 1 \le j \le m) \subseteq K'$, as a model of $T_\forall$.
Since $\epsilon$ is transcendental over $K$, $\Gal(L/K(\epsilon)) \cong (\mathbb{Z}/2\mathbb{Z})^m$.
In particular, for every $S \subseteq \{1 , \ldots, m\}$, there is a field automorphism $\sigma_S \in \Gal(L/K(\epsilon))$ which swaps the square roots of $a_j + \epsilon$ if and only if $j \in S$.  Let $L_S$ be the $T_\forall$-model with underlying field $L$, with the same $\mathcal{L}_i$-structure as $L$ for $i > 1$, and with the $\mathcal{L}_1$-structure obtained by pulling back the $\mathcal{L}_1$-structure of $L$ along $\sigma_S$.  If $\Delta$ denotes symmetric difference of sets, then
\begin{align*}
 \{j : L_S \models \phi_2(\sqrt{a_j + \epsilon})\} &= \{j : L \models \phi_2(\sqrt{a_j + \epsilon})\} \\
 \{j : L_S \models \phi_1(\sqrt{a_j + \epsilon})\} &= \{j : L \models \phi_1(\sigma_S(\sqrt{a_j + \epsilon}))\} \\
                                                   &= \{j : L \models \phi_1(\sqrt{a_j + \epsilon})\} ~ \Delta ~ S,
\end{align*}
where the last equality holds because $L \models \phi_1(\sqrt{a_j + \epsilon}) \iff \neg \phi_1(-\sqrt{a_j + \epsilon})$.
Now let $K_S$ be a model of $T$ extending $L_S$.  Since $L_S$ is a model of $T_\forall$ extending $K$, $K_S \succeq K$.  Also,
\begin{align*} \{ j : K_S \models \psi(a_j + \epsilon)\} &= \{j : L_S \models \phi_1(\sqrt{a_j + \epsilon})\} \Delta \{j : L_S \models \phi_2(\sqrt{a_j + \epsilon})\}
\\ &= \{j : L \models \phi_1(\sqrt{a_j + \epsilon})\} \Delta \{j : L \models \phi_2(\sqrt{a_j + \epsilon})\} \Delta S 
\\ &= \{j : K' \models \psi(a_j + \epsilon)\} \Delta S
\end{align*}
Therefore, by choosing $S = S_0 ~ \Delta ~ \{j : K' \models \psi(a_j + \epsilon)\}$, we can arrange that
\[ \{j : K_S \models \psi(a_j + \epsilon)\} = S_0,\]
i.e., $K_S \models \psi(a_j + \epsilon)$ if and only if $j \in S_0$.  Taking $b$ to be $\epsilon \in K_S$, this completes the proof.
\end{proof}
Because $T$ has the independence property and clearly has the strict order property, the best classification-theoretic property we could hope for $T$ to have is NTP$_2$.

\subsection{NTP$_2$ holds}

First we make some elementary remarks about relative algebraic closures.
\begin{lemma}\label{two-fields}
Let $M$ be a pure field.  Let $K$ be a subfield of $M$ which is relatively separably closed in $M$ (in the sense of Definition~\ref{nearly}).  Let $a$ and $b$ be two tuples from $M$ such that $a \forkindep^{\ACF}_K b$, i.e., $a$ and $b$ are algebraically independent from each other over $K$.  Then $K(a)$ is relatively separably closed in $K(a,b)$.
\end{lemma}
\begin{proof}
Embed $M$ into a monster model $\mathfrak{M} \models \ACF$.  By the remarks after Definition~\ref{nearly}, $\tp(a/K)$ and $\tp(b/K)$ are stationary.  Since $a \forkindep_K b$, the type of $b$ over $\acl(K(a))$ is $K$-definable.  Now suppose that some singleton $c \in K(a,b)$ is algebraic over $K(a)$.   Write $c$ as $f(a,b)$, for some rational function $f(X,Y) \in K(X,Y)$.  Note that $\stp(b/K(a))$ includes the statement $f(a,x) = c$.  On the other hand, it does not include $f(a,x) = c'$ for any conjugate $c' \ne c$ of $c$ over $K(a)$.  As $\stp(b/K(a))$ is definable over $K$, $ac$ and $ac'$ cannot have the same type over $K$.  But if $c$ and $c'$ are conjugate over $K(a)$, then $ac$ and $ac'$ have the same type.  So $c'$ does not exist, and $c$ has no other conjugates over $K(a)$.  Thus $c \in \dcl(K(a))$.  So $c^{p^k} \in K(a)$ for some $k$.  As $c$ was an arbitrary element of $K(a,b) \cap K(a)^{alg}$, we see that $K(a)$ is relatively separably closed in $K(a,b)$.
\end{proof}

\begin{lemma}\label{three-fields}
Let $M$ be a pure field.  Suppose $K_0 \subseteq K_1 \subseteq K_2$ are three subfields of $M$, each relatively separably closed in $M$.
Let $c$ be tuple from $M$, possibly infinite.  Suppose that $c \forkindep^{\ACF}_{K_0} K_2$, i.e., $K_2$ and $c$ are algebraically independent over $K_0$.  Then $K_1(c)$ is relatively separably closed in $K_2(c)$.
\end{lemma}
\begin{proof}
As in the previous lemma, embed $M$ into a monster model $\mathfrak{M}$ of ACF.  Then $c \forkindep_{K_0} K_2$, and by properties of forking, $K_1(c) \forkindep_{K_1} K_2$.  By the previous lemma, $K_1(c)$ is relatively separably closed in $K_2K_1(c) = K_2(c)$.
\end{proof}

Now we return to existentially closed fields with valuations and orderings.  As always, $T$ is the model companion.
\begin{lemma}\label{horror}
In a monster model of $T$, let $B$ be a small set of parameters and $a_1, a_2, \ldots$ be a $B$-indiscernible sequence.  Suppose that $B = \acl(B)$ and $a_i = \acl(Ba_i)$ for any/every $i$.  Suppose also that $a_j \forkindep^{\ACF}_B a_{< j}$ for every $j$, i.e., the sequence is algebraically independent over $B$.  Let $c$ be a finite tuple and suppose that $a_1, a_2, \ldots$ is quantifier-free indiscernible over $cB$, i.e., if $i_1 < \cdots < i_m$ and $j_1 < \cdots < j_m$, then
\[ \qftp(a_{i_1} a_{i_2} \cdots a_{i_m} / cB) = \qftp(a_{j_1} a_{j_2} \cdots a_{j_m} / cB).\]
Let $\phi(x;y)$ be a formula over $B$ such that $\phi(c;a_1)$ holds.  Then $\bigwedge_{j = 1}^\infty \phi(x;a_j)$ is consistent.
\end{lemma}
\begin{proof}
Because $a_1, a_2, \ldots$ is $B$-indiscernible, it suffices to show for each $k$ that $\bigwedge_{j = 1}^\infty \phi(x;a_j)$ is not $k$-inconsistent.

First observe that whether or not $c \forkindep^{\ACF}_B a_j$ holds depends only on the quantifier-free type of $c$ and $a_j$ over $B$.  In particular, it does not depend on $j$, by quantifier-free indiscernibility of $a_1, a_2, \ldots$ over $cB$.  If $c \centernot\forkindep^{\ACF}_B a_j$ for one $j$, then this holds for all $j$.  As the $a_j$ are an algebraically independent sequence over $B$, this contradicts the fact that finite tuples have finite preweight in ACF.  So $c \forkindep^{\ACF}_B a_j$ for each $j$.  The same argument applied to the sequence $a_1 a_2, a_3 a_4, \ldots$ shows that $c \forkindep^{\ACF}_B a_1 a_2$.  Similarly $c \forkindep^{\ACF}_B a_1 a_2 a_3$, and so on, and so $c \forkindep^{\ACF}_B a_1 a_2 a_3 \ldots$.

Let $M$ be the monster model of $T$.  Any subset of $M$ closed under
$\acl(-)$ is relatively algebraically closed in $M$, hence relatively
separably closed in $M$.  In particular, if we let $K_0 = B$, $K_1 =
B(a_j) = a_j$, and $K_2 = \acl(B a_1 a_2 \ldots)$, then each of $K_0,
K_1, K_2$ is relatively separably closed in $M$, and $K_0 \subseteq
K_1 \subseteq K_2$.  By the previous paragraph and
Corollary~\ref{acl-char}, $c \forkindep^{\ACF}_B K_2$, so by
Lemma~\ref{three-fields}, we conclude that $K_1(c)$ is relatively
separably closed in $K_2(c)$, i.e., $B(a_j,c)$ is relatively separably
closed in $K_2(c)$.  Using bars to denote perfect closures, this means
that $\overline{B(a_j,c)}$ is relatively algebraically closed in
$\overline{K_2(c)}$.

Recall the function $P(-,-)$ from Theorem~\ref{dream-thm}.  By the ``extension invariance'' part of that theorem,
\[ P(\phi(c;a_j);\overline{B(a_j,c)}) = P(\phi(c;a_j);\overline{K_2(c)}).\]
Now by quantifier-free indiscernibility of $a_1,a_2,\ldots$ over $cB$, we see that $B(a_j,c) \cong B(a_{j'},c)$ for all $j, j'$.  By the isomorphism-invariance part of Theorem~\ref{dream-thm},
\[ P(\phi(c;a_j);\overline{B(a_j,c)}) = P(\phi(c;a_{j'});\overline{B(a_{j'},c)})\]
for all $j, j'$.  Consequently, $P(\phi(c;a_j);\overline{K_2(c)})$ does not depend on $j$.

Now $M$ is a model of $T$ extending $\overline{K_2(c)}$, and in $M$, $\phi(c;a_1)$ holds.  So by the ``density'' part of Theorem~\ref{dream-thm}, $P(\phi(c;a_1);\overline{K_2(c)})$ is some positive number $\epsilon > 0$.  Consequently, $P(\phi(c;a_j);\overline{K_2(c)}) = \epsilon > 0$ for every $j$.

Suppose for the sake of contradiction that $\bigwedge_{j = 1}^\infty \phi(x;a_j)$ is $k$-inconsistent for some $k$.  Let $N$ be big enough that $N\epsilon > k$.  Let $\psi(x)$ be the statement over $K_2$ asserting that at least $k$ of $\phi(x;a_1), \ldots, \phi(x;a_N)$ hold.  By the Keisler measure part of Theorem~\ref{dream-thm}, $P(\psi(c);\overline{K_2(c)}) > 0$, and there is a model $M'$ of $T$ extending $\overline{K_2(c)}$ in which $\psi$ holds.  In particular, $M' \models \exists x : \psi(x)$.  But $K_2$ is relatively algebraically closed in $M$, hence satisfies axiom A1 of \S\ref{axiom-section} by Corollary~\ref{a1-rel-closed}.  By Corollary~\ref{qe-version-1}, the statement $\exists x :  \psi(x)$ holds in $M$ if and only if it holds in $M'$.  Consequently, it holds in $M$, and therefore $\bigwedge_{j = 1}^n \phi(x;a_j)$ is not $k$-inconsistent.
\end{proof}

Recall from \cite{Adler} or \cite{Ch} that the burden of a partial type $p(x)$ is the supremum of $\kappa$ such that there is an inp-pattern in $p(x)$ of depth $\kappa$, that is, an array of formula $\phi_i(x;a_{ij})$ for $i < \kappa$ and $j < \omega$, and some $k_i < \omega$ such that the $i$th row $\bigwedge_{j < \omega} \phi_i(x;a_{ij})$ is $k_i$-inconsistent for each $i$, and such that for any $\eta : \kappa \to \omega$, the corresponding downwards path $\bigwedge_{i < \kappa} \phi_i(x;a_{i,\eta(i)})$ is consistent with $p(x)$.  A theory is NTP$_2$ if every partial type has burden less than $\infty$.  A theory is strong if every partial type has burden less than $\aleph_0$, roughly.  (See \cite{Adler} for a more precise statement.)  At any rate, if every partial type has burden less than $\aleph_0$, then the theory is strong.  By the submultiplicativity of burden (Theorem 11 in \cite{Ch}), it suffices to check the burden of the home sort.

\begin{fact}\label{superadditivity} If $D$ and $E$ are definable sets, $\bdn(D \times E) \ge \bdn(D) + \bdn(E)$.
\end{fact}
In fact, if $\phi_i(x;a_i)$ is an inp-pattern for $D$ and $\psi_j(y;b_j)$ is an inp-pattern for $E$, then $\{\phi_i(x;a_i)\} \cup \{\phi_j(y;b_j)\}$ is an inp-pattern for $D \times E$.

In NIP theories, burden is the same thing as dp-rank, which is known to be additive \cite{dp-add}.  The theories ACVF, $p$CF, and RCF are all known to be dp-minimal, i.e., to have dp-rank 1 \cite{dpExamples}.  One of the descriptions of dp-rank is that a partial type $\Sigma(x)$ over a set $C$ has dp-rank $\ge \kappa$ if and only if there are $\kappa$-many mutually indiscernible sequences over $C$ and a realization $a$ of $\Sigma(x)$ such that each sequence is not indiscernible over $Ca$.

\begin{theorem}\label{thm-ntp2}
The model companion $T$ is NTP$_2$, and strong.  In fact, the burden of affine $m$-space is exactly $mn$, where $n$ is the number of valuations and orderings.
\end{theorem}
\begin{proof}
To show that the burden of $\mathbb{A}^m$ is at least $mn$, it suffices by Fact~\ref{superadditivity} to show that $\bdn(\mathbb{A}^1) \ge n$.  In the case where every $T_i$ is ACVF, one can take $\phi_i(x;y)$ to assert that $v_i(x) = y$, for $1 \le i \le n$, and take $a_{i,0}, a_{i,1}, \ldots$ to be an increasing sequence in the $i$th valuation group.  Variations on this handle the remaining cases.  We leave the details as an exercise to the reader.

For the upper bound, suppose for the sake of contradiction that there is an inp-pattern $\{\phi_i(x;a_{ij})\}_{i < mn+1;~0 \le j < \omega}$ of depth $mn+1$, with $x$ a tuple of length $m$.  We may assume that the $a_{ij}$ form a mutually $\emptyset$-indiscernible array.  Extend the array to the left, i.e., let $j$ range over negative numbers.  Let $B$ be $\acl(a_{ij} : j < 0)$.  From stability theory, one knows that $a_{ij} \forkindep^{\ACF}_B a_{i0} a_{i1} \cdots a_{i,j-1}$ for every $j$.  By mutual indiscernibility, each sequence $a_{i0}, a_{i1}, \ldots$ is indiscernible over $\{a_{ij} : j < 0\}$, hence over $B$.  In particular, $a_{ij} \equiv_B a_{ij'}$ for $j \ne j'$.  For each $i < mn+1$, let $b_{i0}$ be an enumeration of $\acl(B a_{i,0})$.  For $j > 0$, choose $b_{i,j}$ such that $a_{i,j}b_{i,j} \equiv_B a_{i,0}b_{i,0}$.  Then $b_{i,j}$ is an enumeration of $\acl(B a_{i,j})$ for every $i$ and every $j \ge 0$.  Let $c_{i,j}d_{i,j}$ be a mutually $B$-indiscernible array modeled on the array $a_{i,j}b_{i,j}$.  Then $c_{i,j}d_{i,j} \equiv_B a_{i,0}b_{i,0}$, so $d_{i,j}$ is an enumeration of $\acl(B c_{i,j})$.  Also, because $a_{i,0}, a_{i,1}, \ldots$ was already $B$-indiscernible, we must have
\[ c_{i,0} c_{i,1} \cdots \equiv_B a_{i,0} a_{i,1} \cdots\]
for each $i$.  Consequently, $c_{i,j} \forkindep_B^{\ACF} c_{i,0} \cdots c_{i,j-1}$.  And since $d_{i,j} \subseteq \acl(B c_{i,j})$, we also have
\[ d_{i,j} \forkindep_B^{\ACF} d_{i,0} d_{i,1} \cdots d_{i,j-1},\]
using Corollary~\ref{acl-char}.  As $b_{i,0}$ is an enumeration of $\acl(B a_{i,0})$, the elements of $a_{i,0}$ must actually appear somewhere in $b_{i,0}$.  Let $\pi_i$ be the coordinate projection such that $\pi_i(b_{i,0}) = a_{i,0}$.  Hence $c_{i,j} = \pi_i(d_{i,j})$.

Because the $a_{ij}$ formed a mutually $\emptyset$-indiscernible array, the collective type of all the $c_{i,j}$'s must agree with that of all the $a_{i,j}$'s.  Hence $\phi_i(x;c_{i,j})$ is still an inp-pattern of depth $mn+1$.  Let $\psi_i(x;y)$ be $\phi_i(x;\pi_i(y))$.  Then $\psi_i(x;d_{i,j})$ is an inp-pattern of depth $mn+1$.  Let $c$ be a realization of $\bigwedge_{i < mn+1} \psi_i(x;d_{i,0})$.  Note that $c$ is a tuple of length $m$.

Let $\mathfrak{M}$ be the ambient monster.  For each $1 \le k \le n$, let $\mathfrak{M}_k$ be a model of $T_k$ extending $\mathfrak{M} \restriction \mathcal{L}_k$.  By quantifier-elimination, the array $\{d_{i,j}\}$ is still mutually $B$-indiscernible in $\mathfrak{M}_i$.  By additivity of dp-rank and by dp-minimality of the home sort in $\mathfrak{M}_k$, we know that the dp-rank of $\tp(c/B)$ in $\mathfrak{M}_k$ is at most $m$.  In particular, for each $1 \le k \le n$, at most $m$ of the rows in the array $\{d_{i,j}\}$  can fail to be $Bc$-indiscernible in $\mathfrak{M}_k$.  By the pigeonhole principle, there must be some value of $i$ such that the sequence $d_{i,0}, d_{i,1}, \ldots$ is $Bc$-indiscernible in each of $\mathfrak{M}_1, \mathfrak{M}_2, \ldots, \mathfrak{M}_n$.  Back in $\mathfrak{M}$, this means that $d_{i,0}, d_{i,1}, \ldots$ is quantifier-free $Bc$-indiscernible.  Since $d_{i,0}, d_{i,1}, \ldots$ is also $B$-indiscernible and $B$-independent, Lemma~\ref{horror} applies.  Consequently, $\bigwedge_{j = 0}^\infty \psi_i(x;d_{i,j})$ is consistent, because $\psi_i(c;d_{i,0})$ holds.  This contradicts the fact that $\{ \psi_i(x;d_{i,j})\}$ is an inp-pattern.
\end{proof}

\begin{corollary}\label{mir-ntp2}
  If $(K,v_1,\ldots,v_n)$ is an algebraically closed field with $n$
  independent non-trivial valuations, then $(K,v_1,\ldots,v_n)$ is
  strong of burden $n$.
\end{corollary}
\begin{proof}
  Theorems~\ref{miracle} and \ref{thm-ntp2}.
\end{proof}
\section{Forking and Dividing}\label{forking}

We will make use of the following general fact, which is the implication (ii)$\implies$(i) in Proposition 4.3 of \cite{udi-anand}.\footnote{Hrushovski and Pillay assume NIP, but the assumption is unused for the implication (ii)$\implies$(i).}
\begin{fact}\label{keisler-lascar-forking}
Let $\mathfrak{M}$ be a monster model of some theory, let $S \subseteq \mathfrak{M}$ be a small set, and let $\phi(x)$ be a formula with parameters from $\mathfrak{M}$.  Suppose there is a global Keisler measure $\mu$ which is Lascar-invariant over $S$, and suppose $\mu(\phi(x)) > 0$.  Then $\phi(x)$ does not fork over $S$.
\end{fact}

Now we specialize to the theory $T$ under consideration.
\begin{lemma}\label{apply-probabilities}
Let $\mathfrak{M}$ be a monster model of $T$.  Let $S$ be a small subset of $\mathfrak{M}$, and let $p$ be a complete quantifier-free type over $\mathfrak{M}$ which is Lascar-invariant over $S$.  Then there is a Keisler measure $\mu$ on $S(\mathfrak{M})$, Lascar-invariant over $S$, whose support is exactly the set of completions of $p$.
\end{lemma}
This is nothing but a restatement or special case of Theorem~\ref{dream-thm}.
\begin{proof}
Let $a$ be a realization of $p$ in some bigger model, and consider the structure $\mathfrak{M}[a]$ generated by $\mathfrak{M}$ and $a$.  The structure of $\mathfrak{M}[a]$ is determined by $p$.  Also, if $\sigma$ is any Lascar strong automorphism of $\mathfrak{M}$ over $S$, then $p = \sigma(p)$.  This implies that there is a uniquely determined automorphism $\sigma'$ of $\mathfrak{M}[a]$ extending $\sigma$ on $\mathfrak{M}$ and fixing $a$.

Let $\overline{\mathfrak{M}[a]}$ denote the perfect closure of the field of fractions of $\mathfrak{M}[a]$.  This is uniquely determined (as a model of $T_\forall$) by $\mathfrak{M}[a]$, and hence is determined by $p$.  Let $\mu$ be the Keisler measure on $\mathfrak{M}$ which assigns to an $\mathfrak{M}$-formula $\phi(x;b)$ the value
\[ P(\phi(a;b);\overline{\mathfrak{M}[a]}),\]
where $P$ is as in Theorem~\ref{dream-thm}.
By the Keisler measure part of Theorem~\ref{dream-thm}, this is a Keisler measure on the space of completions of $\qftp( \mathfrak{M}[a])$.  By model completeness, any extension of $\qftp(\mathfrak{M}[a])$ to a complete type must satisfy $\tp(\mathfrak{M})$, so we have a legitimate Keisler measure on the space of extensions of $p$ to complete types over $\mathfrak{M}$.  And if $\sigma$ is any Lascar strong automorphism over $S$, then by the ``isomorphism invariance'' part of Theorem~\ref{dream-thm},
\[ P(\phi(a;\sigma(b));\overline{\mathfrak{M}[a]}) = P(\phi(\sigma'(a);\sigma'(b));\overline{\mathfrak{M}[a]}) = P(\phi(a,b);\overline{\mathfrak{M}[a]})\]
where $\sigma'$ is the aforementioned automorphism of $\mathfrak{M}[a]$ extending $\sigma$ and fixing $a$.  Thus $\mu(\phi(x;b)) = \mu(\phi(x;\sigma(b)))$.  We conclude that $\mu(\phi(x;b)) = \mu(\phi(x;b'))$ for any formula $\phi(x;y)$ and any $b, b' \in \mathfrak{M}$ having the same Lascar strong type over $S$.  Finally, if $b$ is a tuple from $\mathfrak{M}$ and $\phi(a;b)$ is a formula which is consistent with $p$, then $\phi(a;b)$ is also consistent with the diagram of $\mathfrak{M}[a]$, hence has positive probability by the ``density'' part of Theorem~\ref{dream-thm}.
\end{proof}

\begin{corollary}\label{every-extension}
Let $\mathfrak{M}$ be a monster model of $T$ and $S$ be a small subset of $\mathfrak{M}$.  Suppose $q$ is a complete quantifier-free type on $\mathfrak{M}$ which is Lascar invariant over $S$.  Then every complete type on $\mathfrak{M}$ extending $q$ does not fork over $S$.
\end{corollary}
\begin{proof}
Let $p(x)$ be a complete type extending $q(x)$.  Let $\phi(x)$ be any formula from $p(x)$.  Let $\mu$ be the Keisler measure from Lemma~\ref{apply-probabilities}.  Then $\mu$ is Lascar invariant over $S$, and $\mu(\phi(x))$ is positive because $\phi(x)$ is consistent with $q(x)$.  By Fact~\ref{keisler-lascar-forking}, $\phi(x)$ does not fork over $S$.
\end{proof}

If $M$ is a model of $T$ and $A, B, C$ are subsets of $M$, let $A \forkindep^{T_i}_C B$ indicate that $A \forkindep_C B$ holds in any/every model of $T_i$ extending $M \restriction \mathcal{L}_i$.

\begin{lemma}\label{qftp-extender}
Work in a monster model $\mathfrak{M}$ of $T$.  Let $a$ be a finite tuple, and $B$ and $C$ be sets (in the home sort, as always).  Suppose $C = \acl(C)$.  Suppose $a \forkindep^{T_i}_C B$ holds for every $1 \le i \le n$.  Then $\qftp(a/BC)$ can be extended to a quantifier-free type $q(x)$ on $\mathfrak{M}$ which is Lascar invariant over $C$.
\end{lemma}
\begin{proof}
Let $V$ be the variety over $C$ of which $a$ is a generic point.  By Fact~\ref{cb-stationarity-variety}, $V$ is geometrically irreducible.

Let $\mathfrak{M}_i$ be a model of $T_i$ extending $\mathfrak{M} \restriction \mathcal{L}_i$.  Within $\mathfrak{M}_i$, $a \forkindep^{T_i}_C B$.  By Adler's characterization of forking in NIP theories (Proposition 2.1 in \cite{udi-anand}), there is an $\mathcal{L}_i$-type $p_i(x)$ on $\mathfrak{M}_i$ which extends the type of $a$ over $BC$ and which is Lascar-invariant over $C$.  The restriction of this $\mathcal{L}_i$-type to a quantifier-free $\mathcal{L}_{rings}$-type must say that $x$ lives on $V$ and on no $\mathfrak{M}_i$-definable proper subvarieties of $V$.  This follows from Lemma~\ref{stronger-than-acf}.  Let $q_i(x)$ be the set of quantifier-free $\mathcal{L}_i$-statements in $p_i(x)$ with parameters from $\mathfrak{M}$.  Then $q_i(x)$ is a complete quantifier-free $\mathcal{L}_i$-type on $\mathfrak{M}$.  Let $q(x)$ be $\bigcup_{i = 1}^n q_i(x)$.  This is a complete quantifier-free type on $\mathfrak{M}$; it is consistent because the $q_i(x)$ all have the same restriction to the language of rings, namely, the generic type of $V$.  Also, $q(x)$ extends $\qftp(a/BC)$, because the $\mathcal{L}_i$-part of $\qftp(a/BC)$ is present in $p_i(x)$ and $q_i(x)$.

To show Lascar-invariance of $q(x)$ over $C$, it suffices to show that if $I$ is a $C$-indiscernible sequence in $\mathfrak{M}$, $a$ and $a'$ are two elements of $I$, and $\phi(x;y)$ is a quantifier-free formula, then $\phi(x;a) \in q(x)$ if and only if $\phi(x;a') \in q(x)$.  In fact, we only need to consider the case where $\phi(x)$ is a quantifier-free $\mathcal{L}_i$-formula, for some $i$.  But then
\[ \phi(x;a) \in q(x) \iff \phi(x;a) \in p_i(x) \iff \phi(x;a') \in p_i(x) \iff \phi(x;a') \in q(x)\]
where the middle equivalence follows from the fact that $p_i(x)$ is Lascar-invariant, and $I$ is $C$-indiscernible within $\mathfrak{M}_i$ (by quantifier-elimination in $T_i$).  Thus $q(x)$ is Lascar-invariant over $C$, as claimed.
\end{proof}

\begin{theorem}\label{forking-is-dividing}
Forking and dividing agree over every set (in the home sort).
\end{theorem}
\begin{proof}
First we show that if $a$ is a finite tuple and $B$ is a set, then $\qftp(a/B)$ does not fork over $B$.  By Lemma~\ref{qftp-extender}, there is a global quantifier-free type $q(x)$ which is Lascar-invariant over $B$.  By Corollary~\ref{every-extension}, any extension of $q(x)$ to a complete global type does not fork over $B$.  So $\qftp(a/B)$ has a global non-forking extension.  Now if $a$ is any small tuple, and $B$ is a set, then $\qftp(a/B)$ does not fork over $B$, by compactness.  Consequently, if $a$ is a small tuple and $B$ is a (small) set, then $\qftp(a'/B)$ does not fork over $B$, where $a'$ enumerates $\acl(aB)$.  By Corollary~\ref{qe-version-1}, $\qftp(a'/B)$ implies $\tp(a'/B)$, so $\tp(a'/B)$ does not fork over $B$.  By monotonicity, $\tp(a/B)$ does not fork over $B$.  As $a$ and $B$ are arbitrary, every set in the home sort is an extension base for forking in the sense of \cite{CK}, so by Theorem 1.2 in \cite{CK}, forking and dividing agree over every set in the home sort.
\end{proof}

\begin{lemma}\label{not-good}
Let $\mathfrak{M}$ be a monster model of $T$ and $C = \acl(C)$ be a small subset of $\mathfrak{M}$.  Suppose $p(x)$ is a complete type on $C$ and $q(x)$ is a complete quantifier-free type on $\mathfrak{M}$, with $q(x)$ extending the quantifier-free part of $p(x)$.  Suppose $q(x)$ is Lascar-invariant over $C$.  Then $q(x) \cup p(x)$ is consistent.
\end{lemma}
\begin{proof}
Let $\mathfrak{M}[a]$ be the structure obtained by adjoining a realization $a$ of $q(x)$ to $\mathfrak{M}$.  Let $W$ be the variety over $\mathfrak{M}$ of which $a$ is the generic point.  By Fact~\ref{cb-stationarity-variety}, $W$ is geometrically irreducible.  Moreover, the ACF-theoretic code $\ulcorner W \urcorner$ for $W$ must lie in $\mathfrak{M}$.  By Lascar invariance of $q(x)$, one sees that $W$ is Lascar invariant over $C$.  Consequently, the finite tuple $\ulcorner W \urcorner$ is fixed by every Lascar strong automorphism over $C$.  So $\ulcorner W \urcorner \subseteq \acl(C) = C$.  Consequently, in an ambient model of ACF we have $Cb(\stp(a/\mathfrak{M})) \subseteq C$, and so $a \forkindep^{\ACF}_C \mathfrak{M}$.  By Lemma~\ref{two-fields}, $C[a]$ is relatively algebraically closed in $\mathfrak{M}[a]$.

Because the quantifier-free type of $a$ over $C$ is consistent with $p(x)$, there is a model $N \models T$ extending $C[a]$ such that within $N$, $\tp(a/C) = p(x)$.  By Lemma~\ref{precise-amalgamation}, we can amalgamate $\mathfrak{M}[a]$ and $N$ over $C[a]$.  So there is a model $N'$ of $T$ extending $N$ and $\mathfrak{M}[a]$.  In $N$, $\tp(a/C) = p(x)$.  As $N \preceq N'$, $\tp(a/C) = p(x)$ holds in $N'$ as well.  And as $N' \supseteq \mathfrak{M}(a)$, $\qftp(a/\mathfrak{M}) = q(x)$.  So $q(x) \cup p(x)$ is consistent.
\end{proof}

\begin{lemma}\label{first-half}
Work in a monster model $\mathfrak{M}$ of $T$.  Let $a$ be a finite tuple, and $B$ and $C$ be sets (in the home sort, as always).  Suppose $a \forkindep^{T_i}_C B$ holds for every $1 \le i \le n$.  Then $a \forkindep_C B$.
\end{lemma}
\begin{proof}
A type forks/divides over $C$ if and only if it forks/divides over $\acl(C)$, so it suffices to show that $\tp(a/BC)$ does not fork over $\acl(C)$.  By monotonicity, it suffices to show that $\tp(a/\acl(BC))$ does not fork over $\acl(C)$.
By Claim 3.6 in \cite{CK} and Lemma~\ref{fork-div-1} above, $a \forkindep^{T_i}_{\acl(C)} \acl(CB)$ for every $i$.  So we may assume that $C = \acl(C) \subseteq B = \acl(B)$.

Now by Lemma~\ref{qftp-extender}, there is a global quantifier-free type $q(x)$ extending $\qftp(a/BC) = \qftp(a/B)$, with $q(x)$ Lascar-invariant over $C$.  Clearly $q(x)$ is also Lascar-invariant over $B$, so by Lemma~\ref{not-good}, $q(x)$ is consistent with $\tp(a/B)$.  Let $p(x)$ be a global complete type extending $q(x) \cup \tp(a/B)$.  Then $p(x)$ does not fork over $C$ by Corollary~\ref{every-extension}.
\end{proof}

Let $\qftp^i(a/B)$ denote the quantifier-free $\mathcal{L}_i$-type of $a$ over $B$, and let $\qftp^{\ACF}(a/B)$ denote the field-theoretic quantifier-free type of $a$ over $B$.
\begin{lemma}\label{craziness}
Let $\mathfrak{M}$ be a monster model of $T$, and let $C = \acl(C)$ be a small subset.  For each $i$, let $M_i$ be a model of $T_i$ extending $\mathfrak{M} \restriction \mathcal{L}_i$.  For each $i$, let $a_i$ be a tuple in $M_i$.  Suppose that $\qftp^{\ACF}(a_i/C)$ does not depend on $i$.  Then we can find a tuple $a$ in $\mathfrak{M}$ such that $\qftp^i(a/C) = \qftp^i(a_i/C)$ for every $i$.
\end{lemma}
\begin{proof}
Let $C[a_i]$ denote the subring or subfield of $M_i$ generated by $C$ and $a_i$.  By assumption, $C[a_i]$ is isomorphic to $C[a_{i'}]$ as a ring, for every $i$ and $i'$.  Use these isomorphisms to identify all the $C[a_i]$ with each other, getting a single ring $C[a]$ which is isomorphic to $C[a_i]$ for every $i$.  Use these isomorphisms to move the $(T_i)_\forall$ structure from $C[a_i]$ to $C[a]$.  Now $C[a]$ is a model of $T_\forall$, and $\qftp^i(a/C) = \qftp^i(a_i/C)$, for every $i$.  As $C = \acl(C)$, $C$ is relatively separably closed in $\mathfrak{M}$, so by Lemma~\ref{precise-amalgamation}, one can embed $C[a]$ and $\mathfrak{M}$ into a bigger model of $T$.  By model completeness and saturation, $\tp(a/C)$ is already realized in $\mathfrak{M}$.
\end{proof}

\begin{lemma}\label{second-half}
Let $a, B, C$ be small subsets of a monster model $\mathfrak{M} \models T$.  Suppose $a \centernot \forkindep^{T_1}_C b$.  Then $a \centernot \forkindep_C b$.
\end{lemma}
\begin{proof}
By Claim 3.6 in \cite{CK} applied to both $T_1$ and $T$, we may assume
$C = \acl(C)$ and $B = \acl(BC)$.  By finite character of forking, we
may assume $a$ is finite.  For every $i$, let $\mathfrak{M}_i$ be an
even more monstrous model of $T_i$ extending $\mathfrak{M}
\restriction \mathcal{L}_i$.  Then $a \centernot \forkindep_C B$ holds
in $\mathfrak{M}_1$.  By Lemma~\ref{fork-div-1}, some
$\mathcal{L}_1$-formula $\phi(x;B)$ in $\tp(a/BC)$ divides over $C$.
By quantifier-elimination in $T_i$, we may assume that $\phi(x;y)$ is
a quantifier-free $\mathcal{L}_1$-formula.  By
Lemma~\ref{dubious-ist}, there is a sequence $B = B^1_0, B^1_1, B^1_2,
\ldots$ in $\mathfrak{M}_1$ which is indiscernible over $C$ and
algebraically independent over $C$, and such that $\bigwedge_{j =
  0}^\infty \phi(x;B^1_j)$ is $k$-inconsistent in $\mathfrak{M}_1$,
for some $k$.  Thus $\qftp^1(B^1_j/C) = \qftp^1(B/C)$, and in a
certain sense
\[ \qftp^{\ACF}(B^1_0 B^1_1 B^1_2 \cdots/C) = \qftp^{\ACF}(B/C) \otimes \qftp^{\ACF}(B/C) \otimes \cdots.\]
The right hand side makes sense because $C$ is relatively separably closed in $B$ (Definition~\ref{nearly}), so $\qftp^{\ACF}(B/C)$ is stationary.

Meanwhile, for $i > 1$, we can apply Lemma~\ref{stupid} to $\mathfrak{M}_i$ and $\tp(B/C)$, getting a sequence $B = B^i_0, B^i_1, B^i_2, \ldots$ which is indiscernible over $C$ and algebraically independent over $C$.  (Note that Lemma~\ref{stupid} is true even without the restriction that $B$ be finite.)  So again, we get $\qftp^i(B^i_j/C) = \qftp^i(B/C)$, and
\[ \qftp^{\ACF}(B^i_0 B^i_1 B^i_2 \cdots/C) = \qftp^{\ACF}(B/C) \otimes \qftp^{\ACF}(B/C) \otimes \cdots.\]
In particular, $\qftp^{\ACF}(B^i_0 B^i_1 B^i_2 \cdots/C)$ does not depend on $i$, as $i$ ranges from $1$ to $n$.  By Lemma~\ref{craziness}, we can therefore find a sequence $B_0, B_1, \ldots$ in $\mathfrak{M}$ such that
\[ \qftp^i(B_0 B_1 \ldots/C) = \qftp^i(B^i_0 B^i_1 B^i_2 \ldots/C)\]
for every $i$.  In particular, $\qftp^i(B_j/C) = \qftp^i(B^i_j/C) = \qftp^i(B/C)$.  Because this holds for all $i$, $\qftp(B_j/C) = \qftp(B/C)$.  Because $B = \acl(B)$, $\qftp(B/C) \vdash \tp(B/C)$ by Corollary~\ref{qe-version-1}.  So $\tp(B_j/C) = \tp(B/C)$ for every $j$.  Also,
\[ \qftp^1(B_0 B_1 \ldots/C) = \qftp^1(B^1_0 B^1_1 \ldots / C)\]
implies that there is an automorphism $\sigma$ of $\mathfrak{M}_1$ sending $B^1_0 B^1_1 \ldots$ to $B_0 B_1 \ldots$.  Consequently, $\bigwedge_{j = 0}^\infty \phi(x;B_j)$ is $k$-inconsistent in $\mathfrak{M}_1$.  Clearly it is also $k$-inconsistent in $\mathfrak{M}$, because $\mathfrak{M}$ is smaller than $\mathfrak{M}_1$.  Since $B_0$, $B_1$ is a sequence of realizations of $\tp(B/C)$, we conclude that $\phi(x;B)$ divides over $C$, in $\mathfrak{M}$.
\end{proof}

\begin{theorem}\label{thm-char}
Let $M$ be a model of $T$, and let $A, B, C$ be subsets of $M$ (in the home sort).  The following are equivalent:
\begin{itemize}
\item $A \forkindep_C B$, i.e., the type of $A$ over $BC$ does not fork over $C$.
\item The type of $A$ over $BC$ does not divide over $C$.
\item $A \forkindep^{T_i}_C B$ for every $1 \le i \le n$.
\end{itemize}
\end{theorem}
\begin{proof}
The first two bullet points are equivalent by Theorem~\ref{forking-is-dividing}.  If $A \forkindep_C B$, then by Lemma~\ref{second-half} $A \forkindep^{T_1}_C B$.  Similarly, $A \forkindep^{T_i}_C B$ for every $1 \le i \le n$.  Conversely, if $A \forkindep_C^{T_i} B$ for every $1 \le i \le n$, then by Lemma~\ref{first-half}, $a \forkindep_C B$ for every finite subset $a \subseteq A$.  By finite character of forking, $A \forkindep_C B$.
\end{proof}

\begin{acknowledgment}
  The author would like to thank Tom Scanlon, who read an earlier
  draft of this paper appearing in the author's PhD dissertation.

  This material is based upon work supported by the National Science
  Foundation under Grant No.\ DGE-1106400 and Award No.\ DMS-1803120.
  Any opinions, findings, and conclusions or recommendations expressed
  in this material are those of the author and do not necessarily
  reflect the views of the National Science Foundation.  
\end{acknowledgment}

\bibliographystyle{plain}
\bibliography{mybib}{}

\end{document}